\documentclass{amsart}

\usepackage{amssymb,amsmath,amscd,amsthm,enumerate,color,verbatim}
\usepackage{enumitem}
\usepackage{graphicx}
\usepackage{subfigure}
\usepackage{tikz}

\DeclareFontFamily{U}{mathx}{\hyphenchar\font45}
\DeclareFontShape{U}{mathx}{m}{n}{
      <5> <6> <7> <8> <9> <10>
      <10.95> <12> <14.4> <17.28> <20.74> <24.88>
      mathx10
      }{}
\DeclareSymbolFont{mathx}{U}{mathx}{m}{n}
\DeclareFontSubstitution{U}{mathx}{m}{n}
\DeclareMathAccent{\widecheck}{0}{mathx}{"71}
\DeclareMathAccent{\wideparen}{0}{mathx}{"75}

\def\cold#1{\textcolor{blue}{#1}}

\date{\today}

\newcommand{\A}{{\mathcal A}}
\newcommand{\Hi}{{\mathcal H}}
\newcommand{\Z}{{\mathbb Z}}
\newcommand{\R}{{\mathbb R}}
\newcommand{\C}{{\mathbb C}}

\newcommand{\T}{{\mathbb T}}
\newcommand{\Q}{{\mathbb Q}}
\newcommand{\D}{{\mathbb D}}
\newcommand{\U}{{\mathbb U}}

\newcommand{\PP}{{\mathbb P}}
\newcommand{\E}{{\mathcal E}}

\newcommand{\e}{{\varepsilon}}
\newcommand{\DS}{{\mathcal{DS}}}

\newcommand{\sgn}{{\mathrm{sgn}}}

\newcommand{\hot}[1]{\textcolor{red}{#1}}
\newcommand{\SO}{{\mathbb{SO}}}

\newcommand{\supp}{{\mathrm{supp}}}

\newtheorem{theorem}{Theorem}[section]
\newtheorem{lemma}[theorem]{Lemma}
\newtheorem{prop}[theorem]{Proposition}
\newtheorem{coro}[theorem]{Corollary}

\theoremstyle{definition}
\newtheorem{remark}[theorem]{Remark}
\theoremstyle{definition}
\newtheorem*{defi}{Definition}
\theoremstyle{definition}

\allowdisplaybreaks

\numberwithin{equation}{section}

\renewcommand{\Im}{\mathrm{Im} \, }
\renewcommand{\Re}{\mathrm{Re} \, }
\newcommand{\tr}{\mathrm{tr} }

\newcommand{\ac}{{\mathrm{ac}}}

\newcommand{\pp}{{\mathrm{pp}}}

\newcommand{\eqdef}{\overset{\mathrm{def}}=}

\begin{document}

%\title{Quantum Walks with Quasiperiodic Coins}

\title[The Unitary AMO]{Spectral Characteristics of the Unitary Critical Almost-Mathieu Operator}

\author[J.\ Fillman]{Jake Fillman}

\address{Department of Mathematics, Virginia Tech, 225 Stanger Street, Blacksburg, VA~24061, USA}

\email{fillman@vt.edu}

\thanks{J.\ F.\ was supported in part by NSF grants DMS--1067988 and DMS--1361625.}

\author[D.\ C.\ Ong]{Darren C. Ong}

\address{Department of Mathematics, University of Oklahoma, Norman, OK~73019-3103, USA}

\email{darrenong@math.ou.edu}

\author[Z.\ Zhang]{Zhenghe Zhang}

\address{Department of Mathematics, Rice University, Houston, TX~77005, USA}

\email{zzhang@rice.edu}

\thanks{Z.\ Z.\ was supported in part by AMS-Simons travel grant 2014-2016.}

\begin{abstract}

We discuss spectral characteristics of a one-dimensional quantum walk whose coins are distributed quasi-periodically. The unitary update rule of this quantum walk shares many spectral characteristics with the critical Almost-Mathieu Operator; however, it possesses a feature not present in the Almost-Mathieu Operator, namely singularity of the associated cocycles (this feature is, however, present in the so-called Extended Harper's Model). We show that this operator has empty absolutely continuous spectrum and that the Lyapunov exponent vanishes on the spectrum; hence, this model exhibits Cantor spectrum of zero Lebesgue measure for all irrational frequencies and arbitrary phase, which in physics is known as Hofstadter's butterfly. In fact, we will show something stronger, namely, that all spectral parameters in the spectrum are of critical type, in the language of Avila's global theory of analytic quasiperiodic cocycles. We further prove that it has empty point spectrum for each irrational frequency and away from a frequency-dependent set of phases having Lebesgue measure zero. The key ingredients in our proofs are an adaptation of Avila's Global Theory to the present setting, self-duality via the Fourier transform, and a Johnson-type theorem for singular dynamically defined CMV matrices which characterizes their spectra as the set of spectral parameters at which the associated cocycles fail to admit a dominated splitting.

\end{abstract}

\maketitle

\setcounter{tocdepth}{1}
\tableofcontents

\section{Introduction}

In recent years, quantum walks have been studied extensively; see \cite{AVWW, BGVW, CGMV, CGMV2, CRWAGW, CedWer2015, DFO, DFV, DMY2, J11, J12, JM, K14, KS11, KS14, ST12} for some papers on this subject that have appeared in recent years. In the case of one-dimensional coined quantum walks, a substantial amount of progress has been made by using the beautiful observation of Cantero, Gr\"unbaum, Moral, and Vel\'azquez which relates the unitary time-one maps of such quantum walks with CMV matrices, a class of unitary operators which arise from a separate construction in the theory of orthogonal polynomials on the unit circle (OPUC) \cite{CGMV}. In particular, this connection enables one to translate questions about the dynamics of a quantum walk into questions about the spectral theory of CMV matrices and OPUC, a subject for which an enormous number of powerful tools exist; see, e.g., \cite{S1,S2}.

In the present paper, we will study a particular one-dimensional coined quantum walk with quasi-periodically distributed coins. In general, a quasi-periodic walk is given by choosing a frequency vector $\beta \in \T^d \eqdef \R^d/\Z^d$ and a continuous map $\tau:\T^d \to \U(2)$, the group of $2 \times 2$ unitary matrices; given such a setup, one defines a family of quantum walks indexed by $\T^d$ by defining quantum coins on $\Z$ via
\[
C_n
=
C_{n,\theta}
=
\tau(n\beta+\theta),
\quad
n \in \Z, \; \theta \in \T^d.
\]
In this paper, we will focus on the scenario in which $d = 1$, and the map is the canonical imbedding of the circle into $\U(2)$ which identifies $\T$ with the rotation group $\SO(2)$. The unitary operators that arise from this construction appear in the physics literature, and they are closely related to the self-adjoint model known as the Almost-Mathieu Operator (AMO) \cite{ShiKat2010}, \cite{ShiKat2011}. In fact, the operators under consideration in this paper satisfy a very strong self-duality under the Fourier transform, and so they are most closely related to the \emph{critical} AMO. We will show that these operators share many spectral characteristics with the critical AMO, such as purely singular continuous spectrum supported on a Cantor set of zero Lebesgue measure. In contrast to the walks discussed in this paper and \cite{ShiKat2010,ShiKat2011}, Cedzich \textit{et al} have considered homogeneous quantum walks in quasi-periodic \emph{electric fields} \cite{CRWAGW} and quantum walks whose coins vary quasi-periodically in \emph{time} \cite{CedWer2015}.

\subsection{Quantum Walks on the Integer Lattice} \label{sec:qw}

We now precisely describe quantum walks on $\Z$. A \emph{quantum walk} is described by a unitary operator on the Hilbert space $\mathcal{H} = \ell^2(\Z) \otimes \C^2$, which models a state space in which a wave packet comes equipped with a ``spin'' at each integer site. Here, the elementary tensors of the form $\delta_n \otimes e_\uparrow$, and $\delta_n \otimes e_\downarrow$ with $n \in \Z$ comprise an orthonormal basis of $\Hi$ (where $\{e_\uparrow,e_\downarrow\}$ denotes the canonical basis of $\C^2$). A time-homogeneous quantum walk scenario is given as soon as coins
\begin{equation}\label{e.timehomocoins}
C_{n}
=
\begin{pmatrix}
c^{11}_{n} & c^{12}_{n} \\
c^{21}_{n} & c^{22}_{n}
\end{pmatrix}
\in \U(2), \quad n \in \Z,
\end{equation}
are specified. As one passes from time $t$ to time $t+1$, the update rule of the quantum walk applies the coins coordinatewise and shifts spin-up states to the right and spin-down states the the left, viz.
\begin{align}
\delta_{n} \otimes e_\uparrow & \mapsto
  c^{11}_{n} \delta_{n+1} \otimes e_\uparrow
+ c^{21}_{n} \delta_{n-1} \otimes e_\downarrow , \label{e.updaterule1} \\
\delta_n \otimes e_\downarrow  & \mapsto
  c^{12}_{n} \delta_{n+1} \otimes e_\uparrow
+ c^{22}_{n} \delta_{n-1} \otimes e_\downarrow \label{e.updaterule2}.
\end{align}
If we extend this by linearity and continuity to general elements of $\mathcal{H}$, this defines a unitary operator $U$ on $\mathcal{H}$. Equivalently, denote a typical element $\psi \in \Hi$ by
\begin{equation} \label{eq:typicalelt}
\psi
=
\sum_{n \in \Z} \left( \psi_{\uparrow,n} \delta_n \otimes e_\uparrow + \psi_{\downarrow,n} \delta_n \otimes e_\downarrow \right),
\end{equation}
where one must have
\[
\sum_{n \in \Z}\left( \left| \psi_{\uparrow,n} \right|^2 + \left| \psi_{\downarrow,n} \right|^2 \right)
<
\infty.
\]
We may then describe the action of $U$ in coordinates via
\begin{equation} \label{eq:ucoord:def}
\begin{split}
[U\psi]_{\uparrow,n}
& =
c_{n-1}^{11} \psi_{\uparrow,n-1} + c_{n-1}^{12} \psi_{\downarrow,n-1}, \\
[U\psi]_{\downarrow,n}
& =
c_{n+1}^{21} \psi_{\uparrow,n+1} + c_{n+1}^{22} \psi_{\downarrow,n+1}.
\end{split}
\end{equation}
Next, order the basis of $\mathcal{H}$ by setting $\varphi_{2m} = \delta_m \otimes e_\uparrow$, $\varphi_{2m+1} = \delta_m \otimes e_\downarrow$ for $m \in \Z$
\begin{comment}
 (\hot{It looks like we need to set $\varphi_{2m} = \delta_m \otimes e_\uparrow$, $\varphi_{2m+1} = \delta_m \otimes e_\downarrow$ so that $J$, $C'$ of (2.3), (2.4), and $f$, $g$ of p.7 match the general forms on top part of p.22 in Section 6? Otherwise, there is a index difference. We can also change the indices of $\mathcal L_x$ and $\mathcal M_x$ on p.22. But then they will be different from those in (2.8) of GZ, and we have to add some words to justify our change. I mean, it looks like it easier to make the change above. Then we don't need to change anything else. Zhenghe.}).
 \end{comment}
 In this ordered basis, the matrix representation of $U : \mathcal{H} \to \mathcal{H}$ is given by
\begin{equation}\label{e.umatrixrep}
U
=
\begin{pmatrix}
\ddots &\ddots & \ddots & \ddots &&&&  \\
& 0 & 0 & c_1^{21} & c_1^{22} &&& \\
&c_{0}^{11} & c_{0}^{12} & 0 & 0 &&&  \\
&& & 0 & 0 & c_2^{21} & c_2^{22} &  \\
&& & c_1^{11} & c_1^{12} & 0 & 0 &  \\
&& && \ddots & \ddots &  \ddots & \ddots
\end{pmatrix},
\end{equation}
where e.g. $U_{2,1}=c_0^{12}$, and $U_{1,3} = c_1^{22}$. This can be checked readily using the update rule \eqref{e.updaterule1}--\eqref{e.updaterule2}; compare \cite[Section~4]{CGMV}.

%The reader should be warned that many authors (e.g.\ \cite{CGMV,DFV}) use the \emph{transpose} of the matrix $U$ in \eqref{e.umatrixrep}. However, since we imagine a vector $\psi$ as a column vector with $U$ acting on the left, our representation is more appropriate for the present paper. For quantum dynamical purposes, this is not a major concern for the following reason. If $\E$ is an extended CMV matrix, then $\E^\top = S^* \E' S$, where $S$ denotes the left shift on $\ell^2(\Z)$ and $\E'$ has Verblunsky coefficients $\alpha_n' = \alpha_{n+1}$. Since all of our theorems are stable under shifting the Verblunsky coefficients, it does not matter whether one considers $U$ or $U^\top$ in what follows.

In this paper, we will consider the case of coined quantum walks on the integer lattice where the coins are distributed quasi-periodically according to the rule
\begin{equation} \label{eq:coindef}
C_n
=
C_{n,\beta,\theta}
=
R_{2\pi(n\beta + \theta)},
\quad
n \in \Z,
\end{equation}
where $R_\gamma$ denotes counterclockwise rotation through the angle $\gamma$, i.e.,
\[
R_\gamma
=
\begin{pmatrix}
\cos\gamma & -\sin\gamma \\
\sin\gamma &  \cos\gamma
\end{pmatrix},
\quad
\gamma \in \R.
\]
We will denote the update rule associated to coins of the form \eqref{eq:coindef} by $U_{\beta,\theta}$. We will concentrate on the aperiodic case, in which $\beta \notin \Q$. In this case, translation by $\beta$ is a minimal transformation of $\T$, so there exists a uniform compact set $\Sigma_\beta \subseteq \partial \D$ such that $\sigma(U_{\beta,\theta}) = \Sigma_\beta$ for all $\theta \in \T$. Our main theorem characterizes the spectra and spectral type of $U_{\beta,\theta}$.

\begin{theorem} \label{t:uamospec}

The unitary AMO exhibits purely singular Cantor spectrum at all irrational frequencies and all phases, and exhibits purely singular continuous spectrum at all irrational frequencies and almost all phases. More precisely, for every irrational $\beta \in \T$, the following statements hold true:

\begin{enumerate}[label={\rm \alph*}]
\item \label{t:singspec} \hspace{-.18cm}.\ For all $\theta \in \T$, the spectral type of $U_{\beta,\theta}$ is purely singular, i.e.\ $\sigma_{\ac}(U_{\beta,\theta}) = \emptyset$.

\item \label{t:contspec} \hspace{-.18cm}.\ For Lebesgue almost every $\theta \in \T$, the spectral type of $U_{\beta,\theta}$ is purely singular continuous.

\item \label{t:zmspec} \hspace{-.18cm}.\ $\Sigma_\beta$ is a Cantor set of zero Lebesgue measure.

\end{enumerate}

\end{theorem}

Theorem~\ref{t:uamospec}.\ref{t:zmspec} is a consequence of Theorem ~\ref{t:uamospec}.\ref{t:singspec}, Kotani Theory, and our Theorem~\ref{t.criticalenergy} below which implies that Lyapunov exponent of the associated one-parameter family of dynamical systems (see \eqref{eq:tmdef} for the cocycle map) vanishes on the spectrum. In fact, we will prove something stronger than vanishing Lypapunov exponents on the spectrum in Theorem~\ref{t.criticalenergy}. In the language of the global theory of one-frequency analytic quasiperiodic cocycles, the cocycle dynamics on the spectrum are all of critical type. Let us describe this more carefully.

In \cite{AGlobalActa}, Avila has developed a powerful global theory for $1$-frequency quasiperiodic Schr\"odinger operators with real analytic potentials. Roughly speaking, for a real analytic function $v\in C^\omega(\T,\R)$, one may define a family of operators $H_{\beta,v,\theta}:\ell^2(\Z)\rightarrow\ell^2(\Z)$ (known as \emph{Schr\"odinger operators}) by
\begin{equation}\label{schrodingeroperators}
[Hu]_n
=
u_{n-1} + u_{n+1} + v(\theta + n\beta)u_n,
\quad
u \in \ell^2(\Z), \;
n \in \Z.
\end{equation}
Just like in the present paper, one may associate a one-parameter family of dynamical systems to the operators in \eqref{schrodingeroperators}. Specifically, one may define $ A^{(E-v)}\in C^\omega(\T,\mathrm{SL(2,\R)})$ by
 \begin{equation}\label{schrodinger:cocycle}
A(\theta)=A^{(E-v)}(\theta)
=
\begin{pmatrix}E-v(\theta) & -1\\ 1 & 0\end{pmatrix},
\quad
E \in \R, \; \theta \in \T.
 \end{equation}
Given a map as in \eqref{schrodinger:cocycle}, one may define (positive) iterates via
\[
A_n(\theta)
=
A(\theta + (n-1)\beta)
\cdots
A(\theta + \beta)
A(\theta),\ n>0.
\]
Then \cite{AGlobalActa} classified each energy $E$ in the spectrum into three categories:

\begin{enumerate}
\item \textit{supercritical}, if $\sup_{\theta\in\T}\|A_n(\theta)\|$ grows exponentially (positivity of Lyapunov exponent).
\item \textit{subcritical}, if there is a uniform subexponential bound on the growth of $\|A_n(\zeta)\|$ through some band $|\Im \zeta|<\delta$.
\item \textit{critical}, otherwise.
\end{enumerate}

According to Bourgain-Goldstein \cite{BG}, supercritical energies correspond to pure point spectrum with exponentially localized states (Anderson Localization) for almost every frequency $\beta$. Avila \cite{A0} recently demonstrated that subcritical energies correspond to absolutely continuous spectrum for all irrational frequencies $\beta$ (it was shown for exponentially Liouville frequency in \cite{Aaa}). However, critical energies remain mysterious, in general.

Using his global theory \cite{AGlobalActa}, Avila showed that for any fixed irrational frequency $\beta$, for a (measure theoretically) typical potential $v\in C^\omega(\T,\R)$, the operator $H_{\beta,v,\theta}$ has no critical energy. Thus the occurrence of critical energies is a very rare phenomenon. In fact, before our model in the present paper, the only operators known to display critical energies are the critical Almost Mathieu Operator (i.e. the operator given by \eqref{schrodingeroperators} with potential function $v(x) = 2\cos(2\pi x)$, which is also known as Harper's Model) \cite[Theorem~19]{AGlobalActa}, and the extended Harper's Model \cite[Corollary~5.1]{JM2}. Moreover, our model is the first unitary operator to display critical energies.

The cocycle map associated to the operators $U_{\beta,\theta}$ (defined below in \eqref{eq:tmdef}) is meromorphic, not analytic. However it can be renormalized into an analytic map, which enables one to classify the energies in the spectrum as supercritical, subcritical and critical. Then, our stronger version of vanishing Lyapunov exponents on the spectrum is the following theorem:

\begin{theorem}\label{t.criticalenergy}
For all irrational frequencies $\beta$, the operator family $(U_{\beta,\theta})_{\theta \in \T}$ is of critical type in sense that the associated renormalized matrix cocycle $(\beta,N^z)$ defined by \eqref{eq:tmNdef} is critical for every $z\in \Sigma_\beta$.
\end{theorem}

The structure of the paper is as follows. In Section~\ref{sec:bg}, we collect some useful background information and definitions. Section~\ref{sec:aubry} discusses Aubry duality for the present model. In Section~\ref{sec:spectype}, we discuss the spectral type. We begin by proving Theorem~\ref{t:uamospec}.\ref{t:singspec}. Next, we exploit Aubry duality to prove that the spectral type is purely continuous for all frequencies and almost all phases, which completes the proof of Theorem~\ref{t:uamospec}.\ref{t:contspec}. In Section~\ref{sec:zle}, we use global theory for singular one-frequency quasiperiodic cocycles to prove that the Lyapunov exponent vanishes on the spectrum. When combined with Kotani theory and the singular nature of the spectrum, this yields the conclusion of Theorem~\ref{t:uamospec}.\ref{t:zmspec}, namely, that the spectrum is a Cantor set of zero Lebesgue measure. In Section~\ref{sec:johnson-marx}, we work out an essential ingredient in our application of global theory, namely, a version of the theorem of Johnson for singular CMV matrices which identifies the resolvent set with the set of spectral parameters at which the associated matrix cocycles enjoy a dominated splitting.  To the best of our knowledge, this is only known for nonsingular CMV matrices, i.e., CMV matrices whose Verblunsky coefficients are bounded away from the unit circle \cite{GJ96,DFLY2}. We prove Johnson's theorem in the singular case by adapting the techniques in Marx \cite{marxpreprint} which were developed for singular Jacobi matrices. We expect that this result is of independent interest beyond quasiperiodic quantum walks, and so we formulate it in an appropriately general context.

\section*{Acknowledgements}

We are grateful to A.\ Avila, S.\ Jitomirskaya, and C.\ Marx for sharing their unpublished preprint \cite{AJM}. We are also grateful to D.\ Damanik for suggesting that we show purely singular continuous spectrum, and to S.\ Jitomirskaya for reminding us of the proof of criticality for energies in the spectrum of the critical AMO in \cite{AGlobalActa}, and to both of them for helpful discussions. We are also thankful to M.\ Zinchenko for helpful conversations regarding the use of tools in \cite{GZ06}.

\section{Background}\label{sec:bg}

\subsection{CMV Matrices and the CGMV Connection}

Given a sequence $\{ \alpha_n \}_{n \in \Z}$ of complex numbers where $\alpha_n \in \overline{\D} = \{ z \in \C : |z| \le 1\} $ for every $n \in \Z$, the associated \emph{singular extended CMV matrix}, $\E = \E_\alpha$, is a unitary operator on $\ell^2(\Z)$ defined by the matrix representation
\begin{equation} \label{def:extcmv}
\E
=
\begin{pmatrix}
\ddots & \ddots & \ddots & \ddots &&&&  \\
& \overline{\alpha_0}\rho_{-1} & -\overline{\alpha_0}\alpha_{-1} & \overline{\alpha_1}\rho_0 & \rho_1\rho_0 &&&  \\
& \rho_0\rho_{-1} & -\rho_0\alpha_{-1} & -\overline{\alpha_1}\alpha_0 & -\rho_1 \alpha_0 &&&  \\
&&  & \overline{\alpha_2}\rho_1 & -\overline{\alpha_2}\alpha_1 & \overline{\alpha_3} \rho_2 & \rho_3\rho_2 & \\
&& & \rho_2\rho_1 & -\rho_2\alpha_1 & -\overline{\alpha_3}\alpha_2 & -\rho_3\alpha_2 &    \\
&& && \ddots & \ddots & \ddots & \ddots &
\end{pmatrix},
\end{equation}
where $\rho_n = \left( 1 - |\alpha_n|^2 \right)^{1/2}$ for every $n \in \Z$. We refer to $\{\alpha_n\}_{n \in \Z}$ as the sequence of \emph{Verblunsky coefficients} of $\E$. Note that all unspecified matrix entries are implicitly assumed to be zero.

\bigskip

%For $\alpha \in \D$ and $z \in \partial \D$, the corresponding Szeg\H{o} matrix is defined by
%$$
%S(\alpha,z)
%=
%\frac{1}{\rho} \begin{pmatrix}
%z & - \overline{\alpha} \\
%-\alpha z & 1
%\end{pmatrix},
%\quad
%\rho = \sqrt{1 - |\alpha|^2}.
%$$
%The Szeg\H{o} transfer matrices associated to $\E_\alpha$ are then defined by
%$$
%T(n,m;z)
%=
%\begin{cases}
%S(\alpha_{n-1},z) \cdots S(\alpha_m,z) & n > m \\
%I & n = m \\
%S(\alpha_n,z)^{-1} \cdots S(\alpha_{m-1},z)^{-1} & n < m
%\end{cases}
%$$
%where $n, m \in \Z$.

In \cite{CGMV}, Cantero, Gr\"unbaum, Moral, and Vel\'azquez observed that one may connect CMV matrices and quantum walks; let us briefly describe this connection. If $\alpha_{2n} \equiv 0$, then the CMV matrix in \eqref{def:extcmv} becomes
\begin{equation}\label{e.ecmvoddzero}
\mathcal{E}
=
\begin{pmatrix}
\ddots & \ddots & \ddots & \ddots &&&&  \\
& 0 & 0 & \overline{\alpha_1}  & \rho_1 &&&  \\
& \rho_{-1} & - \alpha_{-1} & 0 & 0 &&&  \\
&&& 0 & 0 & \overline{\alpha_3}  & \rho_3 &  \\
&&& \rho_1 & - \alpha_1 & 0 & 0 &    \\
&&&& \ddots & \ddots &  \ddots & \ddots
\end{pmatrix};
\end{equation}
the matrix in \eqref{e.ecmvoddzero} strongly resembles the matrix representation of $U$ in \eqref{e.umatrixrep}. Note, however, that $\rho_n \ge 0$ for all $n$, so \eqref{e.umatrixrep} and \eqref{e.ecmvoddzero} may not match exactly when $c_n^{kk}$ is not real and nonnegative. However, this can be easily resolved by conjugation with a suitable diagonal unitary; the interested reader is referred to \cite{CGMV} for details.

\bigskip
In fact, in Section~\ref{sec:johnson-marx}, we will define a class of generalized CMV matrices which includes both the standard formalism for CMV matrices and the operators that correspond to the quantum walks considered in the present paper. In particular, for our purposes later, let us also describe another factorization of $U$, which is closely tied to the usual $\mathcal L \mathcal M$ factorization of CMV matrices. More concretely, we have $U = JC'$, where
\begin{equation} \label{eq:Jdef}
J: \delta_n \otimes e_\uparrow
\mapsto
\delta_{n-1} \otimes e_\downarrow,
\quad
J: \delta_n \otimes e_\downarrow
\mapsto
\delta_{n+1} \otimes e_\uparrow
\end{equation}
and
\begin{equation} \label{eq:Cprimedef}
C' :\delta_n \otimes e_\uparrow
\mapsto
c_n^{21} \delta_n \otimes e_\uparrow + c_n^{11} \delta_n \otimes e_\downarrow,
\quad
C' :\delta_n \otimes e_\downarrow
\mapsto
c_n^{22} \delta_n \otimes e_\uparrow + c_n^{12} \delta_n \otimes e_\downarrow.
\end{equation}

\subsection{Transfer Matrices}

Since the present model is unitarily equivalent to a CMV matrix, we may write down $2 \times 2$ matrix cocycle which propagates solutions to the time-independent eigenvalue equation. More precisely, notice that $U=U_{\beta,\theta}$ is a finite-range operator; consequently, the definition \eqref{eq:ucoord:def} extends $U$ to a linear operator on the substantially larger vector space
\[
\C^{\Z} \otimes \C^2
=
\{ \psi: \Z \times\{\uparrow,\downarrow\} \to \C \}.
\]
Thus, we may speak of generalized solutions of the eigenvalue equation $U \psi = z\psi$ where $\psi \in \C^\Z \otimes \C^2$ need not decay at $\pm \infty$. Specifically, suppose $\psi \in \C^{\Z} \otimes \C^2$ satisfies $U_{\beta,\theta} \psi = z \psi$, where $z \in \C$, and define $\Psi: \Z \to \C^2$ by
\[
\Psi(n)
=
\begin{pmatrix}
\psi_{\uparrow,n} \\
\psi_{\downarrow,n-1}
\end{pmatrix}.
\]
Using \eqref{eq:ucoord:def}, we may write the equation $U_{\beta,\theta} \psi = z\psi$ in coordinates:
\begin{align*}
z\psi_{\uparrow,n+1}
& =
\cos(2\pi(n\beta + \theta)) \psi_{\uparrow,n} -\sin(2\pi(n\beta + \theta)) \psi_{\downarrow,n} \\
z\psi_{\downarrow,n-1}
& =
\sin(2\pi(n\beta + \theta)) \psi_{\uparrow,n} + \cos(2\pi(n\beta + \theta)) \psi_{\downarrow,n},
\end{align*}
Solving these equations to get $\psi_{\uparrow,n+1}$ and $\psi_{\downarrow,n}$ in terms of $\psi_{\uparrow,n}$ and $\psi_{\downarrow,n-1}$, we have
\[
\Psi(n+1)
=
M^z(n\beta + \theta)\Psi(n)
\text{ for all } n \in \Z,
\]
where $M^z$ is defined by
\begin{equation} \label{eq:tmdef}
M^z(x)
=
\sec(2\pi x)
\begin{pmatrix}
z^{-1} & -\sin(2\pi x) \\
-\sin(2\pi x) & z
\end{pmatrix},
\quad
x \in \T, \; z \in \C \setminus \{0\}.
\end{equation}
Thus, the matrix cocycle induced by $M^z$ and rotation by $\beta$ on $\T$ will be of critical importance in our spectral analysis of $U$. As noted before, $M^z$ is meromorphic (as a function of $x$), so we will also make use of the renormalized analytic (singular) cocycle map $N^z$ defined by
\begin{equation} \label{eq:tmNdef}
N^z(x)
=
-2i\cos(2\pi x) M^z(x)
=
\begin{pmatrix}
-2iz^{-1} & 2i \sin(2\pi x) \\
2i\sin(2\pi x) & -2iz
\end{pmatrix}
\end{equation}
for $x \in \T$ and $z \in \C \setminus \{0\}$.

\begin{remark}
As a consequence of the CGMV connection, there are several other reasonable families of transfer matrices which encode various aspects of the spectral theory of $U_{\beta,\theta}$. In particular, we may notice that the operator family $U_{\beta,\theta}$ falls into the setting described in Section~\ref{sec:johnson-marx} with $T:x \mapsto x+\beta$ and
\[
f(x)
=
\begin{pmatrix}
0 & 1 \\
1 & 0
\end{pmatrix},
\quad
g(x)
=
\begin{pmatrix}
\sin(2\pi x) &  \cos(2\pi x) \\
\cos(2\pi x) & -\sin(2\pi x)
\end{pmatrix},
\quad
x \in \T.
\]
In particular, using our $T$ and $g$ above, we have
\[
\begin{pmatrix}
C'_{2k-2,2k-2} & C'_{2k-2,2k-1} \\
C'_{2k-1,2k-2} & C'_{2k-1,2k-1}
\end{pmatrix}
=
\begin{pmatrix}
c_{k-1}^{21} & c_{k-1}^{22} \\
c_{k-1}^{11} & c_{k-1}^{12}
\end{pmatrix}
=
g(T^{k-1}\theta).
\]
 So $U_{\beta,\theta} = \E_{\theta-\beta}$ for $\theta \in \T$. Consequently, we may encode spectral characteristics of the family $(U_{\beta,\theta})_{\theta \in \T}$ using the (normalized) Gesztesy--Zinchenko transfer matrices for the family $(\E_x)_{x \in \T}$ (defined in \eqref{eq:gzmatdef}), which are given by
\[
A^z(x)
\eqdef
A_f^z(x) A_g^z(x)
=
\begin{pmatrix}
z^{-1} & -\sin(2\pi x) \\
-\sin(2\pi x) & z
\end{pmatrix}.
\]
Thus, $N^z(x) = -2i A^z(x)$, so the cocycles $(\beta,N^z)$ and $(\beta,A^z)$ share all essential dynamical characteristics; in particular, $(\beta,N^z)$ admits a dominated splitting (defined later) if and only if $(\beta,A^z)$ admits a dominated splitting. Thus, Theorem~\ref{t:johnson} immediately yields the following for our model:
\begin{theorem}\label{t:DS:N}
Fix $\beta \in \T \setminus \Q$, and denote by $\Sigma_\beta$ the common spectrum of $U_{\beta,\theta}$, $\theta \in \T$. Then
\[
\Sigma_\beta
=
\{ z \in \partial \D : (\beta,N^z) \notin \DS \}.
\]
\end{theorem}
\end{remark}

\section{Aubry Duality for the Unitary AMO} \label{sec:aubry}

In this section, we will discuss two versions of Aubry duality for the unitary Almost-Mathieu Operator. The first version expresses a classical version of duality between eigenfunctions and Bloch waves, while the second is a slightly more abstract framework which shows how (direct integrals of) the coin and shift operators are dual to one another via the Fourier transform. For the spectral analysis in the present paper, we only need the first version of duality. However, we explicitly work out the second version for two reasons. First, this second version is clearly a natural analogue of the usual Aubry duality transform, which has been extremely useful in the spectral analysis of the self-adjoint AMO as shown in \cite{GJLS}. Secondly, this version reflects how some authors think about duality in the physics literature. Concretely, this version of duality puts the heuristic discussion in \cite{ShiKat2010} on a firm mathematical foundation.

\bigskip

To describe the first version of duality, let us define the \emph{inverse Fourier transform} of a sequence $u \in \ell^2(\Z)$ by
\[
\widecheck{u}(x)
=
\sum_{n \in \Z} u_n  e^{2\pi i n x},
\quad
x \in \T.
\]
Additionally, for $f \in L^2(\T)$, we denote its \emph{Fourier transform} by
\begin{equation} \label{eq:ftdef}
\widehat f_k
=
\int_\T f(x) e^{-2\pi i k x} \, dx.
\end{equation}
Given $\psi \in \Hi$, let us use \eqref{eq:typicalelt} to view it as a pair of $\ell^2$ sequences $\psi_\uparrow = (\psi_{\uparrow,n})_{n \in \Z} \in \ell^2(\Z)$ and $\psi_\downarrow = (\psi_{\downarrow,n})_{n \in \Z} \in \ell^2(\Z)$. If we apply the inverse Fourier transform to the sequences $\psi_\uparrow$ and $\psi_\downarrow$ separately, we observe the following duality between solutions of $U \psi = z \psi$ and $U^\top \varphi = z \varphi$.

\begin{theorem} \label{t:aubry}
Suppose $\psi \in \Hi$ solves $U_{\beta,\theta} \psi = z \psi$. For a.e.\ $\xi \in \T$, we may define a state $\varphi = \varphi^\xi \in \C^\Z \otimes \C^2$ by
\begin{align*}
\varphi_{\uparrow,n}^\xi
& =
e^{2\pi i n\theta} \left(\widecheck\psi_\uparrow(n\beta + \xi) + i \widecheck\psi_\downarrow(n\beta + \xi) \right) \\
\varphi_{\downarrow,n}^\xi
& =
e^{2\pi i n\theta}\left(i\widecheck\psi_\uparrow(n\beta + \xi) +  \widecheck\psi_\downarrow(n\beta + \xi) \right).
\end{align*}
Then $U_{\beta,\xi}^\top \varphi^\xi = z \varphi^\xi$ for Lebesgue almost-every $\xi \in \T$.
\end{theorem}

To formulate the more abstract version of Aubry duality, we consider the direct integral of the family $U_{\beta,\theta}$ over $\theta \in \T$. More specifically, consider the larger Hilbert space $\widetilde \Hi = L^2(\Z \times \T \times \{\uparrow,\downarrow\})$, and denote by $U_\beta: \widetilde \Hi  \to \widetilde \Hi$ the operator which acts on the fiber over $\theta$ by $U_{\beta,\theta}$, i.e.:
\begin{equation} \label{def:abstract:aubry}
[U_\beta \psi](n,\theta,s)
=
(U_{\beta,\theta} \psi(\cdot,\theta,\cdot))(n,s),
\quad
\psi \in \widetilde \Hi, \;
n \in \Z, \; \theta \in \T, \; s \in \{\uparrow, \downarrow\}.
\end{equation}
In \eqref{def:abstract:aubry}, $\psi(\cdot,\theta,\cdot)$ denotes the element $v \in \ell^2(\Z) \otimes \C^2$ with coordinates $v_{s,n} = \psi(n,\theta,s)$, with $n \in \Z$ and $s \in \{\uparrow, \downarrow\}$. It is not hard to check that $U_\beta$ defines a unitary operator from $\widetilde \Hi$ to itself. Now, denote the  Fourier transform on $\widetilde \Hi$ by $\,\widehat{\cdot}$, that is, for $\varphi \in \widetilde \Hi $, we define
\[
\widehat{\varphi}(n,\theta,s)
=
\sum_{m \in \Z} \int_{\T} e^{-2\pi i m \theta} e^{-2\pi i n x} \varphi(m,x,s) \, dx,
\quad
n \in \Z, \; \theta \in \T, \; s \in \{\uparrow,\downarrow\}.
\]
We then define the Aubry dual operator $\A : \widetilde \Hi  \to \widetilde \Hi$ and a change of coordinates in spin space $\mathcal X: \widetilde \Hi \to \widetilde \Hi$ by
\[
[\A \psi](n,\theta,s) = \widehat{\psi}(n,n\beta+\theta,s),
\quad
\psi \in \widetilde \Hi, \;
n \in \Z, \; \theta \in \T, \; s \in \{\uparrow, \downarrow\}
\]
and
\[
\begin{pmatrix}
[\mathcal X \psi](n,\theta,\uparrow) \\
[\mathcal X \psi](n,\theta,\downarrow)
\end{pmatrix}
=
\frac{1}{\sqrt{2}}
\begin{pmatrix}
i &  1 \\
1 & i
\end{pmatrix}
\begin{pmatrix}
\psi(n,\theta,\uparrow) \\
\psi(n,\theta,\downarrow)
\end{pmatrix},
\quad
\psi \in \widetilde \Hi, \; n \in \Z, \; \theta \in \T.
\]

\begin{theorem} \label{t:aubry:abstract}
For all irrational $\beta \in \T$, one has
\[
\mathcal X\mathcal A U_\beta
=
U_\beta^\top \mathcal A \mathcal X.
\]
\end{theorem}

\begin{proof}[Proof of Theorem~\ref{t:aubry}]
Fix $\beta$ and $\theta$. First, let us recast $U\psi = z\psi$ in coordinates by using \eqref{eq:ucoord:def}:
\begin{align*}
z\psi_{\uparrow,n}
& =
\psi_{\uparrow,n-1} c_{n-1}^{11} + \psi_{\downarrow,n-1} c_{n-1}^{12} \\
z\psi_{\downarrow,n}
& =
\psi_{\uparrow,n+1} c_{n+1}^{21} + \psi_{\downarrow,n+1} c_{n+1}^{22}.
\end{align*}
 Multiply both sides (of the first equation) by $e^{2\pi i n x}$, sum over $n \in \Z$, and re-index the sum to get
\begin{align*}
z \widecheck \psi_\uparrow(x)
& =
\sum_{n\in\Z} \left( c_{n-1}^{11} \psi_{\uparrow,n-1}  + c_{n-1}^{12} \psi_{\downarrow,n-1} \right) e^{2\pi i n x} \\
& =
e^{2\pi i x} \sum_{n\in\Z} \left(\cos(2\pi(n\beta+\theta)) \psi_{\uparrow,n}  - \sin(2\pi(n\beta+\theta)) \psi_{\downarrow,n} \right) e^{2\pi i n x}
\end{align*}
for Lebesgue almost-every $x \in \T$. Applying the exponential definitions of the trigonometric functions and the definition of the inverse Fourier transform, the calculation above yields
\begin{equation} \label{eq:psihatupw}
z\widecheck \psi_\uparrow(x)
=
\frac{e^{2\pi i x}}{2} \left(  e^{2\pi i \theta} w_\uparrow(x+\beta) - ie^{-2\pi i \theta} w_\downarrow(x-\beta) \right)
\end{equation}
for a.e.\ $x \in \T$, where
\[
w_\uparrow(x)
=
\widecheck \psi_\uparrow(x) + i \widecheck \psi_\downarrow(x),
\quad
w_\downarrow(x)
=
i\widecheck \psi_\uparrow(x) + \widecheck \psi_\downarrow(x),
\quad
x \in \T.
\]
Similarly,
\begin{equation} \label{eq:psihatdownw}
z \widecheck \psi_\downarrow(x)
=
\frac{e^{-2\pi i x}}{2}\left( -ie^{2\pi i \theta} w_\uparrow(x+\beta) + e^{-2\pi i \theta}w_\downarrow(x-\beta) \right)
\end{equation}
for a.e.\ $x$. Using \eqref{eq:psihatupw} and \eqref{eq:psihatdownw} in conjunction with the definition of $w$, and the exponential definitions of the trigonometric functions, we get the following for a.e.\ $x \in \T$:
\begin{align}
\label{eq:dual:wup}
z w_\uparrow(x)
= &
\cos(2 \pi x) e^{ 2\pi i \theta} w_\uparrow(x + \beta)
+ \sin(2 \pi x)e^{-2\pi i \theta} w_\downarrow(x - \beta), \\
\label{eq:dual:wdown}
z w_\downarrow(x)
= &
 - \sin(2 \pi x) e^{ 2\pi i \theta} w_\uparrow(x + \beta)
 + \cos(2 \pi x)e^{-2\pi i \theta} w_\downarrow(x - \beta).
\end{align}
Now, let $\varphi =\varphi^\xi \in \C^\Z\otimes\C^2$ be defined as in the statement of Theorem~\ref{t:aubry}, i.e., $\varphi^\xi_{s,n} = e^{2\pi i n \theta} w_s(n\beta + \xi)$ for $n \in \Z$ and $s \in \{\uparrow,\downarrow\}$. Then, with $y = y(\xi,n) = n\beta+\xi$, we have
\begin{align*}
[U^\top_{\beta,\xi} \varphi]_{\uparrow,n}
& =
  \cos(2\pi y) e^{2\pi i (n + 1)\theta} w_{\uparrow}(y+\beta)
+ \sin(2\pi y) e^{2\pi i (n - 1)\theta} w_{\downarrow}(y-\beta) \\
& =
e^{2\pi i n \theta}
(\cos(2\pi y) e^{2\pi i \theta} w_{\uparrow}(y+\beta)
+ \sin(2\pi y) e^{-2\pi i \theta} w_{\downarrow}(y - \beta)
)\\
& =
e^{2\pi i n\theta} zw_{\uparrow}(y) \\
& =
z \varphi_{\uparrow,n}
\end{align*}
for a.e.\ $\xi$ by \eqref{eq:dual:wup}. Similarly, using \eqref{eq:dual:wdown}, we have:
\[
[U^\top_{\beta,\xi} \varphi]_{\downarrow,n}
=
z \varphi_{\downarrow,n},
\quad
\text{a.e.\ } \xi \in \T.
\]
Consequently, $U^\top_{\beta,\xi} \varphi^\xi = z \varphi^\xi$ for a.e.\ $\xi$, as desired.
\end{proof}
Notice that our statement of duality differs slightly from the analogous statement for the critical AMO, which is indeed self-dual. Instead, in our setting, $U$ is dual to $U^\top = U^{-1}$. However, the duality between $U$ and $U^{-1}$ is just as useful as the self-duality exploited in the self-adjoint setting. For instance, as is the case for the critical AMO, an immediate consequence of the duality in Theorem~\ref{t:aubry} is that $U_{\beta,\theta}$ may never have $\ell^1$ eigenfunctions; the companion result for the critical AMO may be found in \cite{De}.

\begin{coro}\label{c:noL1solution}
If $\psi$ is a nonzero solution to $U_{\beta,\theta}\psi = z\psi$ with $z \in \partial \D$, then $\psi \notin \ell^1(\Z,\C^2)$, i.e.,
\[
\sum_{n \in \Z} \left( |\psi_{\uparrow,n}| + |\psi_{\downarrow,n}| \right)
=
\infty.
\]
\end{coro}

\begin{proof}
Suppose on the contrary that $U_{\beta,\theta} \psi = z\psi$ for some nonzero $\psi \in \ell^1(\Z,\C^2)$. Then the series which define the inverse Fourier transforms of $\psi_\uparrow$ and $\psi_\downarrow$ converge uniformly on $\T$, so we may follow the argument in the proof of Theorem~\ref{t:aubry} to see that all of the a.e.\ statements may be replaced by statements which hold uniformly on $\T$. In particular, we have $U^\top_{\beta,\theta} \varphi^\theta = z \varphi^\theta$. Since $\psi \in \ell^1(\Z,\C^2) \setminus \{0\}$, $\widecheck\psi_\uparrow$ and $\widecheck\psi_\downarrow$ are continuous and not identically zero; thus, $\varphi$ is bounded and does not go to zero at $\pm \infty$, which is impossible. To see this, notice that the space of solutions to
\begin{equation} \label{eq:geneigeq}
U_{\beta,\theta} \eta = z\eta,
\quad
\eta \in \C^{\Z} \otimes \C^2
\end{equation}
is two-dimensional, so we may choose a second element $\psi' \in \C^{\Z} \otimes \C^2$ so that $\{\psi,\psi'\}$ spans the space of solutions of \eqref{eq:geneigeq}. Since $|\det(M^z(x))| = 1$ for all $x \in \T \setminus\{1/4,3/4\}$, the modulus of the Wronskian of $(\psi_{n+1,\uparrow},\psi_{n,\downarrow})^\top$ and $(\psi'_{n+1,\uparrow},\psi'_{n,\downarrow})^\top$ does not depend on $n$, and hence:
\[
\lim_{|n|\to \infty} \left\|(\psi'_{n+1,\uparrow}, \psi'_{n,\downarrow})^\top\right\|
=
\infty.
\]
Recall that $U = JC'$ and $U^\top = C'J$, where $J$ and $C'$ are defined by \eqref{eq:Jdef} and \eqref{eq:Cprimedef}, so $U^\top v = zv$ if and only if $v = C'\eta$, where $\eta$ solves \eqref{eq:geneigeq}. Thus, since $C'$ is a direct sum of $2 \times 2$ unitary matrices, all solutions of $U_{\beta,\theta}^\top v = zv$ either decay or blow up at $\pm \infty$. However, $\varphi^\theta$ exhibits neither of these behaviors, which is a contradiction.
\end{proof}

\begin{proof}[Proof of Theorem~\ref{t:aubry:abstract}]

Let $\mathcal X$ be as in the statement of the Theorem, and notice that $\A$ and $\mathcal X$ commute with one another. Using the factorization $U = JC'$, we obtain a factorization $U_\beta = J C'_\beta$, where $C_\beta'$ and $J$ act on $\widetilde \Hi$ as follows:
\begin{align*}
[J \psi](n,\theta,\uparrow)
& =
\psi(n-1,\theta,\downarrow) \\
[J \psi](n,\theta,\downarrow)
& =
\psi(n+1,\theta,\uparrow) \\
[C'_\beta \psi](n,\theta,\uparrow)
& =
\sin(2\pi(n\beta + \theta)) \psi(n,\theta,\uparrow)
+ \cos(2\pi(n\beta + \theta)) \psi(n,\theta,\downarrow) \\
[C'_\beta \psi](n,\theta,\downarrow)
& =
\cos(2\pi(n\beta + \theta)) \psi(n,\theta,\uparrow)
- \sin(2\pi(n\beta + \theta)) \psi(n,\theta,\downarrow).
\end{align*}
Evidently, $C'_\beta$ and $J$ are real-symmetric, so it follows that $U_\beta^\top = C'_\beta J$.
\newline

Using the explicit form of $C_\beta'$, one may perform a direct calculation to see that $\mathcal X^* C_\beta' \mathcal X = Q_\beta$, where
\[
\begin{pmatrix}
[Q_\beta \psi](n,\theta,\uparrow) \\
[Q_\beta \psi](n,\theta,\downarrow)
\end{pmatrix}
=
\begin{pmatrix}
0 & e^{-2\pi i (n\beta + \theta)} \\
e^{2\pi i (n\beta + \theta)} & 0
\end{pmatrix}
\begin{pmatrix}
\psi(n,\theta,\uparrow) \\
\psi(n,\theta,\downarrow)
\end{pmatrix}.
\]
A similar calculation reveals that $\mathcal X^* Q_\beta^\top \mathcal X = C_\beta'$, as well. Next, we compute $J \mathcal A$. Using the definitions, we have
\begin{align*}
[ J\A \psi ](n,\theta,\uparrow)
& =
[ \A \psi ](n-1,\theta,\downarrow) \\
& =
\sum_{m \in \Z} \int_{\T} \! e^{-2\pi i m((n-1)\beta + \theta)} e^{-2\pi i (n-1) x} \psi(m,x,\downarrow) \, dx \\
& =
\sum_{m \in \Z} \int_{\T} \! e^{-2\pi i m(n\beta+\theta)} e^{-2\pi i n x} e^{2\pi i (m\beta + x)} \psi(m,x,\downarrow) \, dx.
\end{align*}
Consequently,
\[
[\A^* J \A \psi](n,\theta,\uparrow)
=
e^{2\pi i (n\beta + \theta)} \psi(n,\theta,\downarrow).
\]
Similarly, one obtains
\[
[\A^* J \A \psi](n,\theta,\downarrow)
=
e^{-2\pi i (n\beta + \theta)} \psi(n,\theta,\uparrow).
\]
Thus, we obtain $\A^* J \A = Q_\beta^\top$. Similar calculations prove that $\A$ conjugates $Q_\beta$ to $J$. More specifically, we may compute
\begin{align*}
[Q_\beta \A \psi ](n,\theta,\uparrow)
& =
e^{-2\pi i (n\beta+\theta)} [ \A \psi ](n,\theta,\downarrow) \\
& =
\sum_{m \in \Z} \int_{\T} \! e^{-2\pi i (n\beta + \theta)} e^{- 2\pi i m(n\beta + \theta)} e^{-2\pi i n x} \psi(m,x,\downarrow)  \, dx \\
& =
\sum_{m \in \Z} \int_{\T} \!  e^{- 2\pi i m(n\beta + \theta)} e^{-2\pi i n x} \psi(m-1,x,\downarrow)  \, dx.
\end{align*}
Consequently, $[\A^* Q_\beta \A \psi](n,\theta,\uparrow) = \psi(n-1,\theta,\downarrow)$. Similarly, $[\A^* Q_\beta \A \psi](n,\theta,\downarrow) = \psi(n+1,\theta,\uparrow)$. Consequently, $\A^* Q_\beta \A = J$. Since $\mathcal X$ and $\mathcal A$ commute, it follows that
\[
\mathcal X^* \A^*  U_\beta \A \mathcal X
=
\mathcal X^* \A^*  J \A \mathcal X \A^*  \mathcal X^* C_\beta'\mathcal X \A
=
\mathcal X^* Q_\beta^\top \mathcal X \A^* Q_\beta \A
=
C_\beta' J
=
U_\beta^\top
\]
as desired.
\end{proof}

\section{The Spectral Type} \label{sec:spectype}

As a pleasant consequence of the duality described in Theorem~\ref{t:aubry}, let us see how it may be used to deduce absence of point spectrum for almost all parameters. Specifically, we will show that the unitary AMO has no point spectrum for all irrational $\beta$ and then for almost every $\theta$. Let us begin by proving that the spectral type of $U_{\beta,\theta}$ is always purely singular. This is a standard ``high-barrier'' argument using results of trace-class scattering theory.

\begin{proof}[Proof of Theorem~\ref{t:uamospec}.\ref{t:singspec}]
 Since $\beta$ is irrational, we may choose integers $n_j \in \Z$ with $ \cdots < n_{-1} < n_0 < n_1 < \cdots$ so that $|c_{n_j}^{11}| = |\cos(2\pi(n_j\beta+\theta))| < 2^{-|j|}$ for all $j \in \Z$. Consequently, it is not hard to see that $U_{\beta,\theta}$ is a trace-class perturbation of the quantum walk update rule $\widetilde U$ given by decoupling at sites in the set $\{ n_j : j \in \Z\}$. Concretely, $\widetilde U$ is given by specifying coins
$$
\widetilde C_{n_j}
=
\begin{pmatrix}
0 & -\sgn(\sin(2\pi(n_j\beta + \theta))) \\
\sgn(\sin(2\pi(n_j\beta + \theta))) & 0
\end{pmatrix},
\quad
j \in \Z,
$$
and $\widetilde C_n = C_n = R_{2\pi(n\beta+\theta)}$ for all other values of $n$. Evidently, $\widetilde U$ is a direct sum of finite-dimensional operators; consequently, $\widetilde U$ has no a.c.\ spectrum. Since trace-class perturbations leave a.c.\ spectrum invariant, the a.c.\ spectrum of $U$ is empty; see, e.g., \cite[Chapter~6]{SimTrace}.
\end{proof}

As we proceed with the argument which excludes eigenvalues, we will need to pass between the cocycles $M^z$ and $N^z$ defined by \eqref{eq:tmdef} and \eqref{eq:tmNdef}. Clearly, to relate these objects, one must have fine control over products of the form
\begin{equation} \label{eq:cosproddef}
2^k\prod_{j=0}^{k-1}|\cos[2\pi(\theta+j\beta)]|.
\end{equation}
For example, using the Birkhorff ergodic theorem and formula \eqref{eq:logcosint} below, one may conclude that
\[
\lim_{k \to \infty} \frac1k\log\left( 2^k\prod_{j=0}^{k-1}|\cos[2\pi(\theta+j\beta)]| \right)
=
0,
\]
for almost every $\theta \in \T$ by ergodicity of $x \mapsto \theta+\beta$. However, this yields rather poor control over the product in \eqref{eq:cosproddef}, so one must go beyond the ergodic theorem. Thus, the following proposition (which is~\cite[Theorem 5.3]{AJM}) is essential to our arguments.

\begin{prop}\label{p:subharmonic:sum}
Let $f:\T\rightarrow\C$ be analytic and $\beta$ a fixed irrational number. Let $\{\frac{p_n}{q_n}\}_{n\in\Z_+}$ be the continued fraction approximants of $\beta$. Then there exists a constant $C = C(\beta) > 0$ and a subsequence $(q_{n_l})$ of $(q_n)$ such that
\[
\sum^{q_{n_l}-1}_{j=0}\log|f(\theta+j\beta)|-q_{n_l}\int_\T\log|f(x)| \, dx
\le
C
\]
uniformly in $\theta\in\T$.
\end{prop}

In particular, this proposition yields the following corollary immediately.

\begin{coro}\label{c:cos:upbound}
Let $\beta\notin\Q$ be as in Propostion~\ref{p:subharmonic:sum}. Then there is a $C_0>0$ and a subsequence $(q_{n_l})$ of $(q_n)$ such that
\begin{equation}\label{eq:subharmonic:prod1}
2^{q_{n_l}}\prod^{q_{n_l}-1}_{j=0}\left|\cos[2\pi(\theta+j\beta)]\right|\le C_0
\end{equation}
uniformly in $\theta\in\T$. Consequently, we obtain for any $r\in\Z_+$,
\begin{equation}\label{eq:subharmonic:prod2}
2^{rq_{n_l}}\prod^{rq_{n_l}-1}_{j=0}\left|\cos[2\pi(\theta+j\beta)]\right|\le C^r_0
\end{equation}
uniformly in $\theta\in\T$.
\end{coro}

\begin{proof}
First, we compute that for $\e \in \R$,
\begin{equation} \label{eq:logcosint}
\int_{\T}\log|\cos2\pi (\theta+i\e)| \, dx
=
2\pi|\e|-\log2.
\end{equation}
%\footnote{We would like to thank Helge Kr\"uger for pointing out this elegant proof of \eqref{eq:logcosint}}
For the proof of the present corollary, we only need the case $\e=0$. However, we will need \eqref{eq:logcosint} for $\e\in\R$ in Section~\ref{sec:zle}. The argument follows \cite[pp.~160--161]{Ahlfors} and is given for the convenience of the reader. We first consider the case $\e\ge 0$. Let $\Gamma$ be the rectanglular contour with vertices $0$, $\pi$, $\pi+i\e$, and $i\e$, oriented counterclockwise. Let $\zeta=\theta+i\e$. It is clear that
$$
h(\zeta)
\eqdef
\log(1-e^{2i \zeta})
=
\log(-2ie^{i\zeta}\sin \zeta)
$$
is analytic on the closed region enclosed by $\Gamma$ except at $\zeta= 0, \pi$. However, by following small circular quadrants around $0$ and $\pi$ and letting the radii of these circles go to zero, we obtain
\[
\int_{\Gamma} \log(-2ie^{i\zeta}\sin \zeta) \, d\zeta
=
0
\]
by Cauchy's Theorem. Since the integrand has period $\pi$, the integrals over the vertical edges of $\Gamma$ cancel one another. The integral over the upper horizontal side tends to $0$ as $\e\rightarrow\infty$. Hence we first get $\int^\pi_0 \log(-2ie^{i\theta}\sin \theta) \, d\theta = 0$, which in turn implies that for all $\e\ge 0$,
$$
\int^{\pi+i\e}_{i\e} \log(-2ie^{i\zeta}\sin \zeta) \,d\zeta
=
\int^{\pi}_{0} \log[-2ie^{i(\theta+i\e)}\sin (\theta+i\e)] \, d\theta=0.
$$
Thus,
\[
\pi \log2
-\frac{\pi^2}{2}i
+\frac{\pi^2}{2}i
-\pi\e
+\int^\pi_0\log[\sin(x+i\e)] \, d\theta
=
0,
\]
which clearly implies
$$
\int^\pi_0\log[\sin(\theta+i\e)] \, d\theta
=
\pi\e-\pi\log2=\int^\pi_0\log|\sin(\theta+i\e)| \, d\theta,
$$
where the last equality follows from $\int^\pi_0\log|\sin(\theta+i\e)|d\theta=\Re \int^\pi_0\log[\sin(\theta+i\e)]\, d\theta$.

Then by a simple change of variables, we obtain the following for all $\e\ge 0$:
\begin{align*}
\int_\T\log|\cos(2\pi(\theta+i\e))| \, d\theta
&=
\int_\T\log|\sin(2\pi(\theta+i\e))| \, d\theta\\
&=
\frac1\pi\int^\pi_0\log |\sin(\theta+i2\pi\e)| \, d\theta\\
&=
2\pi\e-\log2.
\end{align*}
Since $\cos2\pi(\theta-i\e)=\overline{\cos2\pi(\theta+i\e)}$ for all $\theta,\e \in \R$, we obtain \eqref{eq:logcosint} for all $\e\in\R$.

Now let $f(\theta)=\cos(2\pi \theta)$, and choose $n_l$ as in Proposition~\ref{p:subharmonic:sum}. Using Proposition~\ref{p:subharmonic:sum} and the $\e = 0$ case of \eqref{eq:logcosint} we have
\begin{equation}\label{eq:subharmonic:sum2}
q_{n_l}\log 2 + \sum_{j=0}^{q_{n_l}-1}  \log|\cos(2\pi(\theta+j\beta))|
\le
C
\end{equation}
for all $\theta \in \T$. Let $C_0=e^{C}$. By exponentiating both sides of \eqref{eq:subharmonic:sum2}, we then obtain \eqref{eq:subharmonic:prod1}. Let $g(x)$ be the right-hand side of \eqref{eq:subharmonic:prod1}, and $g_r(\theta)$ be right-hand side of \eqref{eq:subharmonic:prod2}. By the uniformity of the estimate \eqref{eq:subharmonic:prod1} in $x\in\T$, we obtain
\[
g_r(\theta)=\prod^{r-1}_{j=0}g(\theta+jq_{n_l})
<
C_0^r
\]
uniformly in $\theta \in \T$, which is nothing other than \eqref{eq:subharmonic:prod2}.

\end{proof}

Finally we need the following Lemma.

\begin{lemma}\label{l:badtheta:zeromeasure}
Let $\mathbf k=\{k_n\}\subset\Z_+$ be an arbitrary sequence of positive integers such that $\lim_{n\rightarrow\infty}k_n=\infty$, and let $\|\cdot\|_{\R/\Z}$ denote the distance to the nearest integer. For each pair $l, m\in\Z_+$, set $E_{m,l}=\{\theta\in\T: \|k_l\theta\|_{\R/\Z}<\frac 1{4m}\}$. Then the following set has zero Lebesgue measure:
$$
E(\mathbf k):=\bigcap_{m\in\Z_+}\bigcup_{N\in\Z_+}\bigcap_{l\ge N}E_{m,l}.
$$
\end{lemma}

\begin{proof}
Let $\mu$ denote normalized Lebesgue measure on $\T$. First, it is clear that $\mu(E_{m,l})=\frac 1{2m}$ for each $m, l\in\Z_+$. Moreover, $E(\mathbf k)\subset \bigcup_{N\in\Z_+}\bigcap_{l\ge N}E_{m,l}$ for all $m,l$. Since
\[
\mu\left(\bigcup_{N\in\Z_+}\bigcap_{l\ge N}E_{m,l}\right)=\lim_{N\rightarrow\infty}\mu\left(\bigcap_{l\ge N}E_{m,l}\right)\le \frac 1{2m}
\]
for each $m\in\Z_+$, it follows that
\[
\mu(E(\mathbf k))\le \mu\left(\bigcup_{N\in\Z_+}\bigcap_{l\ge N}E_{m,l}\right)\le \frac 1{2m}.
\]
Since this holds for all $m \in \Z_+$, $\mu (E(\mathbf k))=0$.
\end{proof}

\begin{defi}

Given $\beta \in \T \setminus \Q$, let $q_{n_l}$ be as in Proposition~\ref{p:subharmonic:sum}, and denote $\mathbf k = \{4q_{n_l}\}_{l =1}^\infty$ and $E = E(\mathbf k)$ as in Lemma~\ref{l:badtheta:zeromeasure}. Define the set $\mathcal S_\beta$ by
\[
\mathcal S_{\beta}
=
\left\{\frac l2 + \frac{k\beta}{2}:\ k, l\in\Z\right\}\bigcup E(\mathbf k).
\]

\end{defi}
Of course, $\mu(S_\beta)=0$ since it is the union of a zero-measure set and a countable set. We are now in a position to prove Theorem~\ref{t:uamospec}.\ref{t:contspec}. Clearly, this will follow from the following theorem:

\begin{theorem}
If $\beta$ is irrational and $\theta \notin \mathcal S_\beta$, then $\sigma_{\pp}(U_{\beta,\theta}) = \emptyset$.
\end{theorem}

\begin{proof}
Singularity of the cocycle map $N^z$ and the different form of self-duality complicate matters considerably and force us to do substantial additional work, but our proof follows the outline of an argument of Avila \cite{AvPSWCC}. Let $\beta$ and $\theta$ be as stated,  abbreviate $U = U_{\beta, \theta}$, and adopt the same notation as the proof of Theorem~\ref{t:aubry}. Suppose for the sake of establishing a contradiction that $\sigma_{\pp}(U_{\beta,\theta}) \neq \emptyset$, and let $0 \neq \psi \in \Hi$ and $z \in \partial \D$ be such that $U \psi = z \psi$. As before, we view $\psi \in \Hi$ as a pair of $\ell^2$ sequences $\psi_\uparrow = (\psi_{\uparrow,n})_{n \in \Z}$ and $\psi_\downarrow = (\psi_{\downarrow,n})_{n \in \Z}$. Let
\[
\widecheck\psi_\uparrow(x)=\sum_{n \in \Z} \psi_{\uparrow,n} e^{2\pi inx},
\quad
\widecheck\psi_\downarrow(x)=\sum_{n \in \Z} \psi_{\downarrow,n} e^{2\pi inx}
\]
be the inverse Fourier Transforms of $\psi_\uparrow$ and $\psi_\downarrow$, respectively, and define $w_\uparrow$ and $w_\downarrow$ as in Section~\ref{sec:aubry}, i.e.,
\[
w_\uparrow(x)
=
\widecheck \psi_\uparrow(x) + i \widecheck \psi_\downarrow(x),
\quad
w_\downarrow(x)
=
i\widecheck \psi_\uparrow(x) + \widecheck \psi_\downarrow(x).
\]
Since $\psi_s \in \ell^2(\Z)$ for $s \in \{\uparrow, \downarrow\}$, it follows that $\widecheck\psi_s$ and $w_s$ are in $L^2(\T)$. From the proof of Theorem~\ref{t:aubry}, we recall \eqref{eq:dual:wup} and \eqref{eq:dual:wdown}, which read:
 \begin{align*}
z w_\uparrow(x)
= &
\sin(2 \pi x)e^{-2\pi i \theta} w_\downarrow(x - \beta) + \cos(2 \pi x) e^{ 2\pi i \theta} w_\uparrow(x + \beta), \\
z w_\downarrow(x)
= &
\cos(2 \pi x)e^{-2\pi i \theta} w_\downarrow(x - \beta) - \sin(2 \pi x) e^{ 2\pi i \theta} w_\uparrow(x + \beta)
\end{align*}
for a.e.\ $x \in \T$. We may rephrase these identities to see that the cocycle which propagates solution data is semi-conjugate to a diagonal cocycle. More precisely, we may rearrange to \eqref{eq:dual:wup} and \eqref{eq:dual:wdown} to see that
\begin{equation} \label{eq:w:cocycleconj}
M^z(x)
\begin{pmatrix}
e^{-2\pi i \theta} w_\downarrow(x-\beta) \\
w_\uparrow(x)
\end{pmatrix}
=
\begin{pmatrix}
w_\downarrow(x) \\
e^{2\pi i \theta} w_\uparrow(x+\beta)
\end{pmatrix}
\end{equation}
for a.e.\ $x \in \T$, where $M^z(x)$ is as defined in \eqref{eq:tmdef}. Next, since $z \in \partial \D$ (by assumption), we have $\bar z = z^{-1}$. Using this, \eqref{eq:dual:wup}, and \eqref{eq:dual:wdown}, we also have
\begin{equation} \label{eq:w:cocycleconj2}
M^z(x)
\begin{pmatrix}
\bar{w}_\uparrow(x) \\
e^{2\pi i \theta} \bar{w}_\downarrow(x - \beta)
\end{pmatrix}
=
\begin{pmatrix}
e^{-2\pi i \theta} \bar{w}_\uparrow(x+\beta) \\
\bar{w}_\downarrow(x)
\end{pmatrix}
\end{equation}
for a.e.\ $x$. Combining \eqref{eq:w:cocycleconj} and \eqref{eq:w:cocycleconj2}, we obtain
\begin{equation}\label{semi-conjugacy}
M^z(x) B(x)
=
B(x + \beta) D_\theta
\text{ for a.e. } x \in \T,
\end{equation}
where
\[
B(x)
=
\begin{pmatrix}
e^{-2\pi i \theta} w_\downarrow(x-\beta) & \bar{w}_\uparrow(x) \\
w_\uparrow(x) & e^{2\pi i \theta} \bar{w}_\downarrow(x-\beta)
\end{pmatrix},
\quad
D_\theta
=
\begin{pmatrix}
e^{2\pi i \theta} & 0 \\
0 & e^{-2\pi i \theta}
\end{pmatrix}.
\]

By taking the determinant of both sides, \eqref{semi-conjugacy} implies that $\det B(x+\beta) = \det B(x)$ for a.e.\ $x \in \T$. Since $\det(B(\cdot)) \in L^1(\T)$, ergodicity of $x\mapsto x+\beta$ on $\T$ implies that $\det B(x)$ is constant almost everywhere; we denote this constant value by $\det(B)$. Equation \eqref{semi-conjugacy} implies that the set $\{x: B(x) = 0\}$ is invariant under $x\mapsto x+\beta$ on $\T$ as well. Thus, we must have $\|B(x)\|>0$ for a.e. $x\in\T$. For, if not, then $\|B(x)\|=0$ for a.e.\ $x\in\T$, which implies $\psi=0$, contrary to our assumption thereupon.

Let us show how the assumption $\theta \notin \mathcal S_\beta$ yields $\det B\neq 0$. Indeed, $\det B(x)=0$ gives
$$
|w_\downarrow(x-\beta)|^2 - |w_\uparrow(x)|^2=0  \text{ for a.e.\ } x\in \T.
$$
Hence, there exists a $\phi(x)$ with $|\phi(x)|=1$ for all $x\in\T$ such that
$$
w_\uparrow(x)
=
\phi(x)e^{2 \pi i \theta} \bar w_\downarrow(x - \beta) \text{ for a.e.\ } x\in\T.
$$
If we write $B(x)=(\vec u(x), \vec v(x))$ where $\vec u(x)$ and $\vec v(x)$ are column vectors, then one readily checks that $\vec u(x)=\phi(x) \vec v(x)$. Hence, \eqref{semi-conjugacy} implies
\[
M(x)\vec u(x)
=
e^{2\pi i \theta} \vec u(x+\beta),
\quad
M(x) \phi(x)\vec u(x)
=
e^{-2\pi i \theta} \phi(x+\beta) \vec u(x+\beta)
\]
for a.e.\ $x$. Since $z$ is fixed, we have dropped it from the notation in the equation above and in the remainder of the argument. Clearly, the above equalities imply
$$
\phi(x+\beta)
=
e^{4\pi i \theta} \phi(x)
\text{ for a.e. } x\in\T.
$$
Now, if we consider the Fourier Transform of $\phi$, defined by \eqref{eq:ftdef}, we have $\phi(x) = \sum_k \widehat\phi_k e^{2\pi i k x}$ (a.e.\ $x$). Thus, we obtain
$$
e^{2\pi ik\beta}\widehat \phi_k
=
e^{4\pi i\theta} \widehat \phi_k \text{ for all } k\in\Z.
$$
$\theta\notin \mathcal S_\beta$ implies $e^{2\pi ik\beta}\neq e^{4\pi i\theta}$ for all $k\in\Z$. Hence $\widehat \phi_k=0$ for all $k$, which forces $\phi(x)=0$ for a.e.\ $x\in\T$, contrary to our assumption that $|\phi(x)|=1$ for a.e.\ $x\in\T$.

Thus, the semi-conjugacy \eqref{semi-conjugacy} becomes a $L^2$ conjugacy between $M(x)$ and the constant diagonal matrix $D_\theta$, to wit:
\begin{equation}\label{conjugacy}
M(x)
=
B(x+\beta)D_\theta B(x)^{-1},
\quad
\text{a.e.\ } x \in \T.
\end{equation}
Now, with $M_k(x)=M(x+(k-1)\beta)\cdots M(x)$ for $k \in \Z_+$ and $x \in \T$, \eqref{conjugacy} gives
$$
M_k(x)
=
B(x+k\beta)D^kB(x)^{-1}
\quad
\text{for all } k \in \Z_+ \text{ and a.e.\ } x \in \T.
$$
Let $\Phi_k(x)=\tr (M_k(x))-2\cos2\pi k\theta$, where $\tr( A)$ denotes the trace of the matrix $A$. By cyclicity of the trace, we obtain
$$
\Phi_k(x)
=
\tr \left( B(x+k\beta)D^kB(x)^{-1} \right) - \tr \left( B(x)D^kB(x)^{-1} \right),
\quad
\forall k \in \Z_+, \text{ a.e.\ } x \in \T.
$$
Since $\|A^{-1}\|=\frac{\|A\|}{\det(A)}$ for any invertible $2 \times 2$ matrix $A$, and $|\tr(E)|\le 2\|E\|$ for any $2\times 2$ matrix $E$, we obtain
\begin{align*}
|\Phi_k(x)|
&=
\left|\tr\left(B(x+k\beta)-B(x))D^kB(x)^{-1} \right)\right|\\
&\le
2 \| (B(x+k\beta)-B(x))D^kB(x)^{-1} \| \\
& \le
\frac{2}{|\det B|}\|B(x+k\beta)-B(x)\|\cdot \|B(x)\|.
\end{align*}
Since $\psi_s \in \ell^2(\Z)$ for $s \in \{\uparrow,\downarrow\}$, it is also clear that $\|B(x)\|\in L^2(\T)$. Thus, by the Cauchy-Schwarz inequality, we obtain
\begin{equation}\label{trace:estimate}
\|\Phi_k\|^2_{L^1}
\le
\frac{4}{|\det(B)|^2}\int_\T\|B(x+k\beta)-B(x)\|^2 \, dx
\cdot \int_\T\|B(x)\|^2 \, dx.
\end{equation}
Note that for any $g\in L^2(\T)$, by Parseval's identity, it holds that
\begin{equation}\label{L2norm:estimate}
\lim_{\e\rightarrow 0}\|g(\cdot+\e)-g\|^2_{L^2}
=
\lim_{\e\rightarrow 0}\sum_{n\in\Z}|\hat g_n|^2|e^{-2\pi\e ni}-1|^2=0.
\end{equation}
Note also for any $2\times 2$ matrix $E=\left(\begin{smallmatrix}a&b\\c&d\end{smallmatrix}\right)$, it holds that $\|E\|\le \|E\|_{\mathrm{HS}}$, where
$$
\|E\|_{\mathrm{HS}}
=
\sqrt{|a|^2+|b|^2+|c|^2+|d|^2}
$$
is the Hilbert-Schmidt norm.

Now by taking any sequence $k_n \in \Z$ such that $\lim_{n\to\infty}\|k_n\beta\|_{\R/\Z}=0$, passing to the Hilbert-Schmidt norm, and using \eqref{L2norm:estimate}, one readily checks that
\begin{equation}\label{eq:L1:liminf0}
\lim_{n\rightarrow\infty}\int_\T\|B(x+k_n\beta)-B(x)\|^2 \, dx
=
0.
\end{equation}
In particular, one may take $k_n=4q_n$, where $q_n$ is the denominator of the continued fraction approximants of $\beta$. Hence, \eqref{trace:estimate} and \eqref{eq:L1:liminf0} imply that
\begin{equation} \label{eq:traceL1decay}
\lim_{n \to \infty} \| \Phi_{4q_n} \|_{L^1}
=
0.
\end{equation}
Now, we will use $N$, defined by \eqref{eq:tmNdef}. Evidently, we can rewrite $N$ as follows:
\begin{align*}\label{trace:extremalcoefficient}
N(x)
& =
\begin{pmatrix}-2iz^{-1} & e^{2\pi i x}-e^{-2\pi i x}\\ e^{2\pi i x}-e^{-2\pi i x} & -2iz
\end{pmatrix},\quad
x \in \T, \; z \in \C \setminus \{0\}..
\end{align*}
A direct computation shows that
\begin{equation}\label{trouble}
\Phi_k(x)=\tr M_k(x)-2\cos(2\pi k\theta)
=
\frac{\tr N_k(x)}{(-2i)^k}\prod^{k-1}_{j=0}\sec[2\pi (x+j\beta)]-2\cos(2\pi k\theta).
\end{equation}
Let
\[
f_k(x)
=
(-2i)^k\prod^{k-1}_{j=0}\cos[2\pi (x+j\beta)],
\quad
k \in \Z_+, \; x \in \T.
\]
Then,  \eqref{trouble} gives
\begin{equation}\label{trouble2}
f_k(x)\Phi_k(x)=\tr N_k(x)-2\cos(2\pi k\theta)f_k(x),
\text{ for all } k \in \Z_+, \; x \in \T.
\end{equation}
First, by a straightforward computation, we notice that $\tr (N_k(x))$ is a trigonometric polynomial of the form
$$
\tr (N_k(x))
=
\sum^{k}_{j=-k}c_n e^{2\pi ijx}.
$$
For $k \in 2\Z_+$, one readily checks that the extremal Fourier coefficient is given by
$$
c_k=2e^{\pi ik(k-1)\beta}.
$$
It is clear that $f_k(x)$ is trigonometric polynomial of degree $k$ as well. Moreover, the extremal Fourier coefficient of $f_k(x)$ is $e^{\pi ik(k-1)\beta}$ whenever $k\in 4\Z_+$. Hence, for any $k \in 4\Z_+$, we obtain for the right hand side of \eqref{trouble2}:
\begin{align*}
\|\tr N_k - 2\cos(2\pi k\theta)f_k\|_{L^1}
& \ge
\left|\int_\T\left(\tr N_k(x)-2\cos(2\pi k\theta)f_k(x)\right)e^{-2\pi ikx} \, dx\right|\\
&=
\left|2e^{\pi ik(k-1)\beta}- 2\cos(2\pi k\theta)e^{\pi ik(k-1)\beta}\right|\\
&=
2-2\cos(2\pi k\theta).
\end{align*}
Now, since $\theta\not\in E(\mathbf k)$ where $\mathbf k=\{4q_{n_l}\}^\infty_{l=1}$, we may choose $m\in\Z_+$ and a subsequence $(k_s)_{s\in\Z_+}$ of $(4q_{n_l})$ such that
$$
\left\|k_s\theta\right\|_{\R/\Z}>\frac1{4m}.
$$
Let $c = 2-2\cos(\frac1{4m})>0$. Then for each $s\in\Z_+$, it holds that
\begin{equation}\label{right:estimate}
\| f_{k_s} \Phi_{k_s} \|_{L^1}
=
\|\tr N_{k_s}-2\cos(2\pi k_s\theta)f_{k_s}\|_{L^1}
\ge
2-2\cos(2\pi k_s\theta)
>
c
>
0.
\end{equation}

On the other hand, since $(k_s)$ is a subsequence of $(4q_{n_l})$, \eqref{eq:traceL1decay}, and \eqref{eq:subharmonic:prod2} with $r=4$ imply the following for the right-hand side of \eqref{trouble2}:
\begin{equation}\label{left:estimate}
\lim_{s\rightarrow\infty}\|f_{k_s}\Phi_{k_s}\|_{L^1}
\le
C_0^4\lim_{s\rightarrow\infty}\|\Phi_{k_s}\|_{L^1}
=
0.
\end{equation}
So \eqref{right:estimate} and \eqref{left:estimate} together implies $0 < c \leq 0$, a contradiction.
\end{proof}

\section{Zero Lyapunov Exponents on and Zero Measure of the Spectrum} \label{sec:zle}

In this section, we prove Theorem~\ref{t.criticalenergy}. As we mentioned in Section~\ref{sec:qw}, the only models known to exhibit critical energies are the critical AMO \cite{AGlobalActa} and the Extended Harper's operator (for certain parameter values) \cite{JM2}. Note that zero Lyapunov exponents on the spectrum for the critical AMO was first proved by Bourgain and Jitomirskaya in \cite{BJ}. Our proof of Theorem~\ref{t.criticalenergy} follows the outline of the argument in \cite{AGlobalActa}. However, our cocycle map is quite different from the Schr\"odinger cocycle map as defined in \eqref{schrodinger:cocycle}. Moreover, as we have seen in Section~\ref{sec:spectype}, our cocycle map is merely meromorphic. These new difficulties again force us to do substantial additional work as in Section~\ref{sec:spectype}. We will start by defining the Lyapunov exponents.

Let $\mathrm{M}_2(\C)$ denote the set of all $2\times 2$ matrices and let $A:\T\rightarrow \mathrm{M}_2(\C)$ be a measurable function satisfying the integrability condition:
\begin{equation}\label{integrability:cocycle}
\int_\T\log\|A(\theta)\| \, d\theta
<
\infty.
\end{equation}
Then we define a skew product of $T:x \mapsto x + \beta$ and $A$ as follows:
\begin{equation}\label{dynamics:cocycle}
(\beta,A):\T\times\C^2\rightarrow \T\times\C^2,\
(\theta,\vec w) \mapsto (\theta+\beta,A(\theta)\vec w).
\end{equation}
$A$ is the so-called \emph{cocycle map}. The $n$th iteration of dynamics will be denoted by $(\beta, A)^n=(n\beta, A_n)$. Thus,
\begin{equation}\label{iteration:cocycle}
A_n(\theta)
=
\begin{cases}
A(\theta+(n-1)\beta) \cdots A(\theta+\beta) A(\theta), &n\ge 1;\\
I_2, &n=0;
% \\ [A_{-n}(\theta+n\beta)]^{-1}, &n\le -1.
\end{cases}
\end{equation}
Note that $A_n(x)=[A_{-n}(\theta+n\beta)]^{-1}$ for $n\le -1$ if all matrices involved are invertible. One of the most important objects in understanding dynamics of $(\beta, A)$ is the \emph{Lyapunov exponent}, which is denoted by
$L(\beta, A)$ and given by
\begin{equation}\label{eq:LE}
L(\beta,A)
=
\lim\limits_{n\rightarrow\infty}\frac{1}{n}\int_{\T}\log\|A_n(\theta)\| \, d\theta
=
\inf_n\frac{1}{n}\int_{\T}\log\|A_n(\theta)\| \, d\theta.
\end{equation}
The limit exists and is equal to the infimum since $\{\int_{\T}\log\|A_n(x)\|  \, dx\}_{n\geq1}$ is a subadditive sequence. If $\beta$ is irrational, then by Kingman's subadditive ergodic theorem, it also holds that
$$
L(\beta, A)
=
\lim\limits_{n\rightarrow\infty}\frac{1}{n}\log\|A_n(\theta)\|,
$$
for Lebesgue almost every $\theta\in\T$.

Back to our setting, as we have seen from previous sections, we will primarily be interested in the following one parameter family of cocycle maps:
\[
M^z: \T\rightarrow \mathrm{M}_2(\C),\
M^z(\theta)
=
\sec(2\pi \theta)
\begin{pmatrix}
z^{-1} & -\sin(2\pi \theta) \\
- \sin(2\pi \theta) & z
\end{pmatrix},\ z\in\partial\D.
\]
The iteration of this cocycle map generates the transfer matrices of our unitary operators $U_{\beta,\theta}$.  To avoid poles $M^z(\theta)$, we consider the renormalized cocycle maps:
\[
N^z(\theta)
=
-2i \cos(2\pi \theta)
M^z(\theta)
=
\begin{pmatrix}
-2iz^{-1} & 2i\sin(2\pi \theta) \\
2i\sin(2\pi \theta) & - 2i z
\end{pmatrix},\ z\in\partial\D,
\]
which is analytic in $\theta$. Let $L(\beta,z)$ be the Lyapunov exponent of the dynamical system $(\beta, M^z)$. Using the formula \eqref{eq:logcosint} for $\e=0$, it is easy to see that
\[
L(\beta,z)
=
L(\beta,M^z)
=
L(\beta,N^z).
\]
for all $z$.  As an immediate corollary of Theorem~\ref{t:uamospec}.\ref{t:singspec}, we may note that $L(\beta,z)$ is positive for Lebesgue almost-every $z \in \partial \D$.

\begin{coro}\label{c:aepositiveLE}
For irrational $\beta$, the Lyapunov exponent $L(\beta,z)$ corresponding to the family $(U_{\beta,\theta})_{\theta \in \T}$ is positive Lebesgue-almost-everywhere on $\partial \D$.
\end{coro}

\begin{proof}
This is immediate from Theorem~\ref{t:uamospec}.\ref{t:singspec} and Kotani Theory for CMV matrices \cite[Section~10.11]{S2}.
\end{proof}

Next, we define the notion of \textit{acceleration} which was introduced by Avila \cite{AGlobalActa}. Let $A\in C^\omega(\T,\mathrm{M}_2(\C))$. Then there exists a $\delta>0$ such that we may extend $A$ to an analytic map on the complex domain $\{\theta+i\varepsilon: \theta\in\T, |\varepsilon|<\delta\}$. Let $A_\varepsilon$ denote the complexified cocycle map $A(\cdot+i\varepsilon)$. Then, the acceleration is defined as
$$
\omega(\beta,A)
=
\lim_{\varepsilon\rightarrow 0+}\frac{1}{2\pi\varepsilon}\left[L(\beta,A_\varepsilon)-L(\beta,A)\right].
$$

The existence of the limit follows from fact that $L(\beta,A_\e)$ is convex in $\e$ which follows from the subharmonicity of $L$ in $\theta+i\e$, see e.g. \cite{AGlobalActa}. The following theorem states that acceleration is quantized.

\begin{theorem}\label{t.quantization_acceleration}
For any irrational frequency $\beta$ and any $A\in C^\omega(\T,\mathrm{M}_2(\C))$, $\omega(\beta,A) \in \frac{1}{2} \Z$. If $A\in C^\omega(\T,\mathrm{SL}(2,\C))$, then $\omega(\beta,A) \in \Z$.
\end{theorem}

This was first proved in \cite{AGlobalActa} for $\mathrm{SL}(2,\C)$-valued cocycles and later extended to $\mathrm{M}_d(\C)$-valued cocycles in \cite{AJS} for all $d\ge 2$. As a consequence, $L(\beta,A_\varepsilon)$ is a piecewise affine function in $\varepsilon$. Denote the upper and lower Lyapunov exponents of $(\beta,A)$ by $L_1(\beta,A)\ge L_2(\beta,A)$, where $L_1(\beta,A)=L(\beta,A)$ (see Section \ref{sec:dominated splitting} for a more detailed description). The following theorem is again from \cite{AJS}.

\begin{theorem}\label{t.ds:affine}
$L_1(\beta,A)>L_2(\beta,A)$ and $L(\beta,A_\varepsilon)$ is affine around $\varepsilon=0$ if and only if $(\beta,A)$ admits a dominated splitting.
\end{theorem}

See Section~\ref{sec:dominated splitting} for the definition of dominated splitting. In light of Theorems~\ref{t.quantization_acceleration} and \ref{t.ds:affine}, we will consider the complexified Lyapunov exponent
\[
L(\beta,z;\e)
\eqdef
L(\beta,N_{\e}^z).
\]

Notice that by Theorem~\ref{t.quantization_acceleration}, $\omega(\beta,N_\e^z)\in \frac12\Z$. However, in order to exclude supercritical energies, we will need to avoid the case of fractional acceleration in our setting. Essentially, this follows from the fact that $M$ is meromorphic with all poles on the real axis.
\begin{lemma} \label{l:accelnothalf}
For all $z \in \partial \D$ and all $\e\in \R$,
\[
\omega(\beta,N_\e^z)
\in \Z.
\]
\end{lemma}
\begin{proof}
Using formula \eqref{eq:logcosint} and the relation between $N^z$ and $M^z$, we clearly have,
\begin{equation}\label{eq:LE:NtoM}
L(\beta,z;\e)
=
L(\beta,M_{\e}^z) + 2\pi|\e|,\mbox{ for all } \e\in \R.
\end{equation}
Notice that for all $\e>0$, $M_\e^z\in C^\omega(\T,\mathrm{SL}(2,\C))$ which implies $\omega(\beta,M_\e^z)\in\Z$ for all $\e\neq 0$. Hence we obtain
\begin{equation} \label{eq:MN:accelrel}
\omega(\beta,N_\e^z)
=
\begin{cases}
\omega(\beta,M_\e^z)+1,\ \e>0,\\
\omega(\beta,M_\e^z)-1,\ \e<0.
\end{cases}
\end{equation}
Consequently, $\omega(\beta,N_\e^z)\in\Z$ for all $\e\neq 0$ by Theorem~\ref{t.quantization_acceleration}. This in turn implies that $\omega(\beta,N^z)=\omega(\beta,N_0^z)\in \Z$ since $L(\beta,z;\e)$ is piecewise linear and continuous in $\e$ for all $\e\in\R$.
\end{proof}

By real symmetricity of $\sin(x)$, we have
$$
N^z(\theta-i\e)
=
-P\cdot\overline{N^z(\theta+i\e)}\cdot P,
\text{ for all } z \in \partial \D,
$$
where
\[
P= \begin{pmatrix} 0 & 1 \\ 1 & 0 \end{pmatrix}.
\]
Hence,  we must have that $L(\beta,z;\varepsilon) = L(\beta,z;-\e)$, i.e., the complexified Lyapunov exponent is even in $\e$.

\begin{theorem}\label{t.le:at:largeheight}
For all $\e$ with $|\e| > 0$ sufficiently large, $L(\beta,z;\e) = 2 \pi |\e|$ uniformly for all $z \in \partial \D$.
\end{theorem}

\begin{proof}
Let's first consider the case $\e>0$. Let $G(\theta;z,\e) \eqdef e^{-2\pi\e}N_{\e}^z(\theta)$. Clearly,
\begin{equation} \label{eq:GLE:rel}
L(\beta,G(\cdot;z,\e))
=
L(\beta,N_\e^z) - 2\pi\e
=
L(\beta,z;\varepsilon) - 2\pi \e.
\end{equation}
Thus, one has
\[
G(\theta;z,\e)
=
\begin{pmatrix}
-2iz^{-1} e^{-2\pi \e} & e^{2\pi i \theta - 4\pi \e} - e^{-2\pi i \theta } \\
e^{2\pi i \theta - 4\pi \e} - e^{-2\pi i \theta } & -2iz e^{-2\pi \e}
\end{pmatrix},
\]
and hence
\[
\lim_{\e \to \infty} G(\theta;z,\e)
=
\begin{pmatrix}
0 & -e^{-2\pi i \theta} \\
- e^{-2\pi i \theta} & 0
\end{pmatrix}.
\]
Clearly, the convergence above holds uniformly in $\theta \in \T$ and in $z\in\partial \D$. Using this, it is easy to see that, for every $\delta > 0$, there is a constant $M = M(\delta)$ such that
\[
(1-\delta) \|\vec v\|
\leq
\| G(\theta;z,\e) \vec v\|
\leq
(1+\delta) \| \vec v\|
\]
for all $\vec v \in \C^2$, all $\theta\in\T$, and all $z\in\partial \D$ whenever $\e  \ge M$. Consequently, $L(\beta,G(\cdot;z,\e)) = 0$ for all $z\in\partial \D$ for large enough $\e$ by convexity and quantization. The result then follows from \eqref{eq:GLE:rel}. Since $L(\beta,z;\e)$ is even in $\e$, we thus obtain the result for all $\e$ with $|\e|$ large.

\end{proof}

\begin{theorem}\label{t.le=-log2}
$L(\beta,z;\e)\ge 2\pi|\e|$ for all $z \in \partial \D$ and all $\e \in \R$. Moreover, the equality holds if and only if $z\in\Sigma_\beta$.
See Figure~\ref{f:complexified:LE} for the graph of $L(\beta,z;\e)$.
\end{theorem}

\begin{proof}
Since $L(\beta,z;\e)$ is both even and convex in $\e$,  by quantization of acceleration, Theorem~\ref{t.le:at:largeheight} implies that for all $\e\ge0$, the acceleration is either $0$, $1/2$, $1$. By Lemma~\ref{l:accelnothalf}, the acceleration cannot be $1/2$. Thus, by Theorem~\ref{t.le:at:largeheight}, we have
$$
L(\beta,z;\e)\ge 2\pi\e
$$
for all $z \in \partial \D$ and all $\e\ge 0$. By Theorem~\ref{t:DS:N}, $z\notin \Sigma$ if and only if $(\beta,N^z)$ admits a dominated splitting. Since $L(\beta, z;\e)$ is even in $\e$, Theorem~\ref{t.ds:affine} then implies that $z\notin \Sigma$ if and only if the acceleration is $0$ around $\e=0$. Thus $z\in \Sigma$ if and only the acceleration is $1$ around $\e=0$. By convexity, quantization of acceleration, and Theorem~\ref{t.le:at:largeheight},
$z\in \Sigma$ if and only if
$$
L(\beta,z;\e) = 2\pi\e, \mbox{ for all } \e\ge0.
$$
Since $L$ is an even function of $\e$, the result for all $\e \in \R$ follows.
\end{proof}

\begin{remark}
Theorem~\ref{t.le:at:largeheight} and the proof of Theorem~\ref{t.le=-log2} together give a precise description of the graph of $L(\beta,z;\e)$ as a function of $\e$ for all $z\in\partial\D$; in turn, this gives us a precise description of the graph of $L(\beta,M^z_\e)$ in $\e$. See Figure~\ref{f:complexified:LE} below for the graphs of $L(\beta,z;\e)$ and $L(\beta,M^z_\e)$ both when $z\in\Sigma_\beta$ and when $z\in\partial\D\setminus\Sigma_\beta$. The graphs of $L(\beta,M^z_\e)$ also illustrate the following: due to the existence of poles of $M^z(\theta)$ on $\T$, one cannot use $L(\beta,M^z_\e)$ to classify energies $z\in\Sigma_\beta$ as supercritical, subcritical, and critical. For instance, $L(\beta,M^z)$ is neither affine nor convex around $\e=0$ for $z\in\partial\D\setminus\Sigma_\beta$.
\end{remark}

\begin{figure}[b]
\begin{center}
\begin{tikzpicture}[yscale=1]
\draw [help lines,->] (-1.6,0) -- (1.6,0);
\draw [help lines,->] (0,-0.5) -- (0,2.2);
\draw [thick,red,domain=-1.5:-0.4] plot (\x, {-1*\x});
\draw [thick,red,domain=0.4:1.5] plot (\x, {1*\x});
\draw [thick,red,domain=-0.4:0.4] plot (\x, {0.4+0*\x});
\draw [thick,blue,domain=-1.5:-0.4] plot (\x, {0*\x});
\draw [thick,blue,domain=0.4:1.5] plot (\x, {0*\x});
\draw [thick,blue,domain=-0.4:0] plot (\x, {1*\x+0.4});
\draw [thick,blue,domain=0:0.4] plot (\x, {-1*\x+0.4});
\node [above] at (0,2.2) {$L$};
\node [above] at (-1.5,1.7) {\hot{$\frac1{2\pi} L(\beta,z;\e)$}};
\node [above] at (-2,0) {\cold{$\frac1{2\pi} L(\beta,M^z_\e)$}};
\node [right] at (1.8,0) {$\e$};
\node[align=left, below] at (0,-0.5)%
{$z\in \partial\D\setminus \Sigma$};
\end{tikzpicture}
\hskip 1.8cm
\begin{tikzpicture}[yscale=1]
\draw [help lines,->] (-1.6,0) -- (1.6,0);
\draw [help lines,->] (0,-0.5) -- (0,2.2);
\draw [thick,red,domain=-1.5:0] plot (\x, {-1*\x});
\draw [thick,red,domain=0:1.5] plot (\x, {1*\x});
\draw [thick,blue,domain=-1.5:1.5] plot (\x, {0*\x});
\node [above] at (0,2.2) {$L$};
\node [above] at (-1.5,1.7) {\hot{$\frac1{2\pi} L(\beta,z;\e)$}};
\node [above] at (-2,0) {\cold{$\frac1{2\pi} L(\beta,M^z_\e)$}};
\node [right] at (1.8,0) {$\e$};
\node[align=left, below] at (0,-0.5)%
{$z\in\Sigma$};
\end{tikzpicture}
\caption{}
\label{f:complexified:LE}
\end{center}
\end{figure}

\vskip .5cm

\begin{proof}[Proof of Theorem~\ref{t.criticalenergy}]
By Theorem~\ref{t.le=-log2}, we have
$$
L(\beta,z;\e)\ge 2\pi|\e|, \mbox{ for all } \e \in \R, \; z \in \partial \D,
$$
and the equality
$$
L(\beta,z;\e)
=
2\pi|\e|, \mbox{ for all }\e \in \R,
$$
holds if and only if $z \in \Sigma$. Hence according to the classification in Section~\ref{sec:qw}, all $z\in\Sigma$ are critical. In particular, $L(\beta,z)=0$ for all $z\in\Sigma$.
\end{proof}

We now have all of the information that we need to prove the remaining part of Theorem~\ref{t:uamospec}, namely, that $\Sigma$ is a Cantor set of zero Lebesgue measure.

\begin{proof}[Proof of Theorem~\ref{t:uamospec}.\ref{t:zmspec}]
Fix $\beta, \theta$ with $\beta$ irrational. By Proposition~\ref{c:aepositiveLE}, the Lyapunov exponent $L(\beta,z)$ is positive for Lebesgue almost every spectral parameter $z\in\partial \mathbb D$, and Theorem~\ref{t.criticalenergy} implies that $L(\beta,z)=0$ on the spectrum. Hence $\sigma(U_{\beta,\theta})$ has Lebesgue measure zero. By ergodicity, $\sigma(U_{\beta,\theta})$ has no isolated points (cf.\ \cite[Theorem~10.16.1]{S2}), so the theorem is proved.
\end{proof}

\section{Johnson's Theorem for Singular CMV Matrices} \label{sec:johnson-marx}

In this section, we will describe a version of Johnson's theorem which holds for singular CMV matrices, following the broad outlines of the corresponding result in \cite{marxpreprint}. We will rely on \cite{GZ06} in this section. In particular, we shall use several theorems in \cite{GZ06} that are asserted in that paper for $z\in \mathbb C\setminus (\partial \mathbb D\cup\{0\})$ and for nonnegative real $\rho$, but are in fact true for $z\in \mathbb C\setminus ( \Sigma\cup\{0\})$ and for complex $\rho$ (see formula \eqref{eq:iSU2} for where $\rho$ appears). Here $\Sigma$ is the spectrum of the CMV matrix. It is straightforward to verify that the theorems that we need and quote apply directly to the more general setting.

\subsection{Dominated splitting and non-degenerate Lyapunov spectrum}\label{sec:dominated splitting}
We first introduce the idea of a dominated splitting, which is a notion from smooth ergodic theory that generalizes uniform hyperbolicity. Let $X$ be a compact metric space, $T:X\rightarrow X$ be a homeomorphism, and $\mu$ be any $T$-ergodic probability measure on $X$.

Recall that $\mathrm{M}_2(\mathbb C)$ refers to the set of $2\times 2$ complex matrices. By \eqref{dynamics:cocycle}, for each measurable $A:X\rightarrow\mathrm{M}_2(\mathbb C)$ there corresponds a dynamical system defined by
\[
(T,A):(x,v)
\mapsto
(Tx,A(x)v).
\]
We denote the iterates of this cocycle by $(T,A)^n = (T^n, A_n)$. By \eqref{iteration:cocycle}, $A_n(x)=A(T^{n-1}x)\cdots A(x)$ for integers $n>0$, and $A_0(x)$ is the identity matrix.

\begin{defi}
We further assume $A\in C(X,\mathrm{M}_2(\C))$. Then, we say $(T,A)$ admits a \emph{dominated splitting} (we write $(T,A)\in\mathcal{D}\mathcal{S}$) if there exists a continuous decomposition $\C^2=E^+(x)\oplus E^-(x)$ for all $x\in X$ with $\mathrm{dim}(E^+(x))=\mathrm{dim}(E^-(x))=1$ satisfying the following two properties:

\begin{enumerate}
\item For all $x\in X$, $A(x)E^{\pm}(x)\subseteq E^\pm(Tx)$.

\item There exists a $N\in\mathbb Z_+$ such that uniformly for all $x\in X$ and for all unit vectors $v^\pm\in E^\pm(x)$, it holds that
      \[
      \lVert A_N(x)v^+\rVert>\lVert A_N(x)v^-\rVert.
      \]
\end{enumerate}

\end{defi}

Equivalently, it is not difficult to see that a cocycle $(T,A)$ admits a dominated splitting if and only if there exists $N\in\mathbb Z_+$ and a continuous change of coordinates $Q\in C(X,\mathrm{GL}(2,\mathbb C))$ such that
\begin{equation}\label{eq:diag:DS}
Q(Tx)^{-1}A(x)Q(x)
=
\begin{pmatrix}
\lambda_1(x)&0\\
0& \lambda_2(x)
\end{pmatrix},
\end{equation}
where $\lambda_1,\lambda_2\in C(X,\mathbb C)$ satisfy
\begin{equation} \label{eq:unif:sep}
\prod^{N-1}_{i=0}\lvert\lambda_1(T^ix)\rvert
>
\prod^{N-1}_{i=0}\lvert\lambda_2(T^ix)\rvert \mbox{ for all } x\in X.
\end{equation}
Much like uniform hyperbolicity, $\DS$ is an open condition in the space $C(X,\mathrm M_2(\C))$ which may be detected by an invariant cone criterion.

For a matrix $A\in\mathrm{M}_2(\C)$, let $\sigma_1(A)\ge \sigma_2(A)$ be the singular values of $A$, i.e., the eigenvalues of $\sqrt{A^*A}$; in particular, $\sigma_1(A) = \|A\|$, the Euclidean operator norm of $A$. Consider a cocycle map $A: X\rightarrow\mathrm{M}_2(\C)$ subject to the integrability condition $\int_X\log\|A(x)\| \, d\mu(x) < \infty$ and consider the dynamical system $(T,A)$. Then for $k=1,2$, the \emph{Lyapunov exponents} of $(T,A)$ are defined by
$$
L_k(T,A)
\eqdef
\lim_{n\rightarrow\infty}\frac1n\int_X\log(\sigma_k(A_n(x))) \, d\mu(x).
$$
Moreover,
$$
L_1(T,A)
=
L(T,A)
:=
\lim_{n\rightarrow\infty}\frac1n\int_X\log\|A_n(x)\| \, d\mu(x),
$$
where $L(T,A)$ is the usual (maximal) Lyapunov exponent, as we have already defined in \eqref{eq:LE}. Then we say that $(T,A)$ has \emph{nondegenerate Lyapunov spectrum} with respect to the ergodic measure $\mu$ if
$$
L_1(T,A)
>
L_2(T,A).
$$
Evidently, if $A(x)\in\mathrm{SL}(2,\C)$ for $\mu$-a.e. $x\in X$, then $L_1(T,A)+L_2(T,A) = 0$. Hence, in this case, $(T,A)$ has nondegenerate Lyapunov spectrum if and only if it has positive Lyapunov exponent, that is, if and only if $L(T,A)>0$.

Finally, for any measurable cocycle map $A$ subject to the condition $\int_X\log\|A(x)\|d\mu<\infty$ which admits a bounded measurable conjugacy to the diagonal cocycle as in \eqref{eq:diag:DS} (i.e. there exists a bounded measurable map $Q:X\rightarrow\mathrm{GL}(2,\C)$ so that \eqref{eq:diag:DS} holds for $\mu$-a.e. $x\in X$), it always holds that
$$
L_k(T,A)
=
\int_X\log|\lambda_k(x)| \, d\mu(x),
\quad
k = 1,2,
$$
where without loss of generality we assume that
\[
 \int_X \log |\lambda_1(x)|d\mu(x) \ge \int_X \log |\lambda_2(x)|d\mu(x).
\]
In particular, if $(T,A)\in \mathcal D \mathcal S$, then it has nondegenerate Lyapunov spectrum with respect to any ergodic measure $\mu$.

\subsection{Generalized CMV matrices}
We will define a class of dynamically defined CMV matrices which are slightly more general than we need. It includes both the standard formalism for CMV matrices and the operators that correspond to the quantum walks considered in the present paper. This class of operators was first introduced in \cite{BHJ}.

Let $f,g$ be continuous functions from $X$ to $i\mathrm{SU}(2)$ (i.e., $\mathrm{U}(2)$ matrices with determinant $-1$). Note that elements of $i\mathrm{SU}(2)$ can be written in the form
\begin{equation}\label{eq:iSU2}
\begin{pmatrix}
\overline \alpha & \rho\\
\overline \rho & -\alpha
\end{pmatrix},
\end{equation}
where $\alpha\in \overline{\D}$ and $|\rho|^2+|\alpha|^2 = 1$. Notice in particular that $\rho$ may be a complex number.

Let us write the matrix elements of $f$ and $g$ as
\[
f(x)
=
\begin{pmatrix}
\overline {\alpha_f(x)} & \rho_f(x)\\
\overline {\rho_f(x)} & -\alpha_f(x)
\end{pmatrix},
\quad
\text{and}
\quad
g(x)
=
\begin{pmatrix}
\overline {\alpha_g(x)} & \rho_g(x)\\
\overline {\rho_g(x)} & -\alpha_g(x)
\end{pmatrix},
\quad
x \in X.
\]
For each $x \in X$, we will define a CMV operator $\mathcal E_x$ by $\mathcal E_x = \mathcal L_x\mathcal M_x$, where
\[
\begin{pmatrix}
(\mathcal L_x)_{2k-1,2k-1} & (\mathcal L_x)_{2k-1,2k}\\
(\mathcal L_x)_{2k,2k-1} & (\mathcal L_x)_{2k,2k}
\end{pmatrix}=f(T^kx),\ k\in\Z
\]
and
\[
\begin{pmatrix}
(\mathcal M_x)_{2k-2,2k-2} & (\mathcal M_x)_{2k-2,2k-1}\\
(\mathcal M_x)_{2k-1,2k-2} & (\mathcal M_x)_{2k-1,2k-1}
\end{pmatrix}=g(T^kx), \ k\in \Z,
\]
and all other entries of $\mathcal L$ and $\mathcal M$ are zero. Compare \cite[(2.6)--(2.9)]{GZ06}; here we essentially set $\alpha_g(x)=-\overline{\alpha_{-1}}$ and $\alpha_f(x)=-\overline{\alpha_{0}}$. We then assume the integrability condition
\begin{equation} \label{eq:logrhoL1}
-\int_X \log| \rho_f | \, d\mu
-\int_X \log| \rho_g | \, d\mu
<
\infty,
\end{equation}
which is the appropriate translation of the condition \eqref{integrability:cocycle} to the present context.

We shall use modified Gesztesy-Zinchenko (GZ) transfer matrices, normalized so the $\rho$ term disappears:
\begin{equation} \label{eq:gzmatdef}
A^z_f(x)
=
\begin{pmatrix}
-\alpha_f(x)& 1\\
1& -\overline{\alpha_f(x)}
\end{pmatrix},\quad
A^z_g(x)
=
\begin{pmatrix}
-\overline{\alpha_g(x)}& z\\
1/z& -\alpha_g(x)
\end{pmatrix}.
\end{equation}
By GZ cocycle recursion, if $x\in X_0$, then $u_x=(u_{x}(n))_{n\in\Z}$ and $v_x=(v_{x}(n))_{n\in\Z}\in \C^\Z$ solve the equations $\mathcal E_x u_x=zu_x$ and $\mathcal M_x u_x=zv_x$ (note that these two equations also imply $\mathcal E^\top_x v_x=zv_x$) if and only if:
\begin{align}
\label{GZrecursionOdd}
&\begin{pmatrix}
u_x(k)\\
v_x(k)
\end{pmatrix}=
\frac{1}{\rho_{g}(T^{\frac{k+1}2}(x))}A_g(T^{\frac{k+1}2}x)\begin{pmatrix}
u_x(k-1)\\
v_x(k-1)
\end{pmatrix}
, \text{$k$ odd.}\\
\label{GZrecursionEven}
&\begin{pmatrix}
u_x(k)\\
v_x(k)
\end{pmatrix}=
\frac{1}{\rho_f(T^{\frac k2}x)}A_f(T^{\frac k2}x)\begin{pmatrix}
u_x(k-1)\\
v_x(k-1)
\end{pmatrix}
, \text{$k$ even,}
\end{align}

Recall that $T$ is a homeomorphism on a compact metric space $X$. Assume from now on that $T$ satisfies the following additional properties:
\begin{enumerate}
\item $T$ is strictly ergodic, i.e., $T$ is minimal and uniquely ergodic, and let $\mu$ be the unique invariant measure.
\item $\supp(\mu) = X$, i.e., $\mu(Y)>0$ for each open set $Y\subset X$.
\item $T$ is an isometry, i.e. $d(Tx,Ty)=d(x,y)$ for all $x$, $y\in X$, where $d$ is the metric on $X$.
\end{enumerate}

Notice that a minimal translation on a compact metrizable group satisfies all these properties, and thus, the setting at hand applies to all CMV operators with Bohr-almost-periodic coefficients. For instance, the minimal translation on the torus $\T^d$ (corresponding to quasiperiodic potentials) or minimal translations on a compact Cantor group (corresponding to limit-periodic operators) are all of this type. We will need these assumptions for the proof of the following Theorem~\ref{t:johnson}. For instance, as we have used in previous sections, it is easy to see that minimality of $T$ and strong operator convergence imply that the spectrum of $\mathcal E_x$, denoted $\Sigma$, is independent of $x\in X$. Then:

\begin{theorem}\label{t:johnson} Let $T$ be a strictly ergodic isometry and suppose that $\supp(\mu)=X$. Then we have
\[
\Sigma
=
\{z\in\partial\D:(T,A^z_fA^z_g)\notin\mathcal D\mathcal S\}.
\]
\end{theorem}
Let us consider the set $\mathcal Z = \{x\in X: \rho_f(x)\rho_g(x)=0\}$. This set is of $\mu$-measure zero by \eqref{eq:logrhoL1}; in particular, it is a nonwhere dense subset of $X$ since we assumed that $\mu$ is positive on any open subset of $X$. Let us define
\[
X_0
=
X\setminus \left( \bigcup_{n\in\mathbb Z} T^n\mathcal Z \right).
\]
In particular $\rho_f(T^kx)\rho_g(T^kx)\neq 0$ for all $k\in\mathbb Z$ when $x\in X_0$.

\subsection{Dominated splitting away from the spectrum}\label{ssec:domawayfromspec}

In this section, we prove the relation $\Sigma \supset\{z\in\partial\D:(T,A^z_fA^z_g)\notin\mathcal D\mathcal S\}$. In fact, we prove a slightly stronger result:
\begin{equation}\label{eq:DSawayspectrum}
\C\setminus (\Sigma\cup\{0\}) \subset \{z\in\C:(T,A^z_fA^z_g) \in \mathcal D\mathcal S\}.
\end{equation}
In other words, for spectral parameters away from the spectrum of the operators $\mathcal E_x$, the dynamical system $(T,A^z_fA^z_g)$ admits a dominated splitting. We will begin by defining the Carath\'eodory and Schur functions. It will turn out that the dominating sections will (very nearly) be given by the Schur and anti-Schur functions. The key idea here is to relate Schur functions to Weyl solutions to the difference equation and then use Combes-Thomas type estimates on the solutions to get domination. The details follow.

Notice that $\E_x$ becomes a direct sum of two half line operators if $\alpha_*(T^mx)\in\partial\D$ for $*=f$ or $g$ and for some $m\in\Z$. Then for all $(x,s)\in X\times[0,2\pi)$, we define the half line unitary operators $\E^{(s)}_{x,j,+}$ on $\ell^2[j,+\infty)$ and $\E^{(s)}_{x,j,-}$ on $\ell^2(-\infty,j]$ via the following formula:
\[
\ \mathcal E_x=
\begin{cases}
\mathcal E^{(s)}_{x,2m-2,-}\oplus \mathcal E^{(s)}_{x,2m-1,+}, &\mbox{ if }\alpha_g(T^mx)=-e^{-is}\in\partial \D,\\
\mathcal E^{(s)}_{x,2m-1,-}\oplus \mathcal E^{(s)}_{x,2m,+}, &\mbox{ if }\alpha_f(T^mx)=-e^{-is}\in\partial \D.
\end{cases}
\]
Let $\mu^{(s)}_{g,x,-}$ be the spectral measures corresponding to $\mathcal E^{(s)}_{x,-2,-}$, $\mu^{(s)}_{g,x,+}$ to $\mathcal E^{(s)}_{x,-1,+}$,  $\mu^{(s)}_{f,x,-}$ to $\mathcal E^{(s)}_{x,-1,-}$, and $\mu^{(s)}_{f,x,+}$ to $\mathcal E^{(s)}_{x,0,+}$.

Notice in any case, $\E_x$ and $\mathcal E^{(s)}_{x,-j,-}\oplus \mathcal E^{(s)}_{x,-j+1,+}$ differ by a finite rank operator. Since finite rank perturbations preserve essential spectrum, we obtain for all $z\notin \Sigma$, either $z$ is in the resolvent set of $\E^{(s)}_{x,j,\pm}$, or it is in the discrete spectrum of $\E^{(s)}_{x,j,\pm}$. Let $\C\PP^1=\C\cup\{\infty\}$ denote the Riemann sphere. We may then define for $*=f$ or $g$ a pair of one-parameter families of functions $m^{(s)}_{*,\pm}:X\times(\C\setminus(\Sigma\cup\{0\})\rightarrow\C\PP^1$ as follows:
\begin{align}
m^{(s)}_{*,-}(x,z)=-\oint_{\partial\mathbb D}\frac{\zeta+z}{\zeta-z}d\mu^{(s)}_{*,x,-}(\zeta) \label{M_-}
\end{align}
and
\begin{align}
m^{(s)}_{*,+}(x,z)=\frac{\mathrm{Re}(1-\alpha_*(x))-i\mathrm{Im}(\alpha_*(x))\oint_{\partial\mathbb D}\frac{\zeta+z}{\zeta-z}d\mu^{(s)}_{*,x,+}(\zeta)}{-i\mathrm{Im}(\alpha_*(x))+\mathrm{Re}(1+\alpha_*(x))\oint_{\partial\mathbb D}\frac{\zeta+z}{\zeta-z}d\mu^{(s)}_{*,x,+}(\zeta)},\label{M_+}
\end{align}
where e.g. $m^{(s)}_{g,-}(x,z)\in \C$ if $z$ is in the resolvent set of $\E^{(s)}_{x,-2,-}$, and $m^{(s)}_{g,-}(x,z)=\infty$ if $z$ is in the discrete spectrum of $\E^{(s)}_{x,-2,-}$.

Then we define functions $\Phi^{(s)}_{*,\pm}(x,z):X\times(\C\setminus(\Sigma\cup\{0\})\rightarrow\C\PP^1$ by
\begin{equation}\label{eq.mtoM}
\Phi^{(s)}_{*,\pm}(x,z)=\frac{1-m^{(s)}_{*,\pm}(x,z)}{1+m^{(s)}_{*,\pm}(x,z)}.
\end{equation}

The functions $\Phi^{(s)}_{*, -}(x,z)$ (resp. $\Phi^{(s)}_{*, +}(x,z)$) are known as the \emph{anti-Schur} (resp.\ \emph{Schur}) functions corresponding to $\mathcal E_{x}$, and $m^{(s)}_{*, -}(x,z)$ (resp.\ $m^{(s)}_{*, +}(x,z)$) are known as the \emph{anti-Carath\'eodory} (resp.\ \emph{Carath\'eodory}) functions corresponding to $\mathcal E_{x}$. Please consult \cite[Appendix~A]{GZ06} for a thorough exposition of relevant facts on these functions. We will provide a brief summary of Schur and anti-Schur functions here for the reader's convenience.

An analytic function $\Phi_+$: $\mathbb D \to\mathbb D$ is called a \emph{Schur function}. We may extend such a function so that it takes $\mathbb C\setminus\overline{ \mathbb D}\to\mathbb C\setminus \overline{\mathbb D}$ via $\Phi_+(w)=\overline{\Phi_+(\overline w^{-1})}^{-1}$. In particular, whenever $\Phi_+$ admits a continuous continuation through some $z\in\partial \D$, it must hold that $\Phi_+(z)\in\partial\D$. We call a function $\Phi_-$ an \emph{anti-Schur} function if $1/\Phi_-$ is a Schur function.

\begin{comment}{A Carath\'eodory function $m_+$ is an analytic function from $\mathbb D$ to the right half plane. We may similarly extend it so that it takes $\mathbb C\setminus \overline{\mathbb D}$ to the left half plane We say that $m_-$ is an anti-Carath\'eodory function if $-m_-$ is a Carath\'eodory function.}

We may associate a Schur (or anti-Schur) function and a Carath\'eodory (or anti-Carath\'eodory) function to a probability measure on the unit circle or a unitary operator, analogous to the way that a Herglotz m-function is associated to a probability measure on the real line or a self-adjoint operator. This relationship between Carath\'eodory or anti-Carath\'eodory function and spectral measure is given by a Poisson transform, for instance, \eqref{M_-}.
\end{comment}

The Schur and Caratheodory functions play an important role in the spectral theory of unitary operators, as their limiting behavior on the unit circle gives us information about the spectral measures. For instance, $\Phi^{(s)}_{g,-}(x,z)=-1$ for some $z\in\partial\D$ if and only if $z$ is a pole of $m^{(s)}_{g,-}(x,z)$, which happens if and only if the spectral measure $\mu^{(s)}_{g,x,-}$ has a pure point at $z$, i.e.\ $z$ is an eigenvalue of $\E^{(s)}_{x,-2,-}$.

Note that if $x\in X_0$ and $z\notin \Sigma\cup\{0\}$, then $\mathcal E_xu=zu$ and  $\mathcal M_x u_x = z v_x$ admit solutions $u_{x,\pm}, v_{x,\pm}$ which are square-summable at $\pm\infty$. In other words, $u_{x,\pm},v_{x,\pm} \in \ell^2(\Z_\pm)$. We define the following:
\begin{equation} \label{eq:mfct:solns-odd}
\Phi_{\pm,x,z}(k)
:=
\frac{v_{x,\pm}(k)}{u_{x,\pm}(k)}, \text{ $k$ odd,}
\end{equation}
\begin{equation} \label{eq:mfct:solns-even}
\Phi_{\pm,x,z}(k)
:=
\frac{u_{x,\pm}(k)}{zv_{x,\pm}(k)}, \text{ $k$ even.}
\end{equation}
Note however that our $u_x$ and $v_x$ come from \cite[(2.141)]{GZ06} rather than from \cite[(2.136)]{GZ06}, which is why our $\Phi$ functions are defined differently from \cite[(2.154)]{GZ06}.

Next, we define the following functions by fixing a particular choice of $k\in\Z$:
\begin{equation} \label{eq:mfct:solns}
\Phi_{g,\pm}(x,z):=\Phi_{\pm,x,z}(-2),\ \Phi_{f,\pm}(x,z):=\Phi_{\pm,x,z}(-1).
\end{equation}
Then, we have
\begin{equation} \label{eq:mfct:solns.alternative}
e^{-is}\Phi^{(s)}_{*,\pm}(x,z)=\Phi_{*,\pm}(x,z),\ *=f,g
\end{equation}

This comes from modifying the calculations in \cite{GZ06}. For simplicity, Gesztesy and Zinchenko always assume that the singular Verblunsky coefficients take on the value $\alpha_g(x)=-1$ to get $\mathcal E_x=\mathcal E^{(0)}_{x,-2,-}\oplus \mathcal E^{(0)}_{x,-1,+}$, (and similarly with $\alpha_f(x)=-1$) but if we dispense with that assumption and allow $\alpha_*(x)=-e^{-is}\in\partial\D$ instead, then this results in factors of $e^{is}$ replacing the $1$ in their (2.51) and (2.54), which leads to the $e^{is}$ factors in their (2.57) and (2.58). Notice however that the relation \eqref{eq:mfct:solns.alternative} only holds for $x\in X_0$.
\begin{comment}
The point is that for any $s\in\R$, the half-line operators $\mathcal E^{(s)}_{x,-2,-}$ and $\mathcal E^{(s)}_{x,-1,+}$ are always well-defined regardless of where are $\alpha_f(T^kx)$ and $\alpha_g(T^kx)$'s (it's just that if $\alpha_g(x)\neq -e^{-is}$, we don't have the decomposition $\mathcal E_x=\mathcal E^{(s)}_{x,-2,-}\oplus\mathcal E^{(s)}_{x,-1,+}$). As a consequence, $m^{(s)}_\pm(x,z)$ and $\Phi^{(s)}_\pm(x,z)$ are well defined for all $(x,z)\in X\times \mathbb (C\setminus (\Sigma\cup 0))$.
\end{comment}

We may then extend the domain of $\Phi_{*,\pm}(x,z)$ to $X\times(\C\setminus (\Sigma\cup \{0\}))$ via $\Phi_{*,\pm}(x,z) = e^{-is}\Phi^{(s)}_{*,\pm}(x,z)$ by arbitrary choice of $s\in [0,2\pi)$. By \eqref{eq:mfct:solns.alternative}, $e^{-is}\Phi^{(s)}_{*,\pm}(x,z)$ is independent of $s$ for all $(x,z)\in X_0\times(\C\setminus (\Sigma\cup \{0\}))$. Notice that $X_0$ is a dense subset of $X$. Hence, $e^{-is}\Phi_{*,\pm}^{(s)}(x,z)$ is independent of $s$ for all $(x,z)\in X\times(\C\setminus (\Sigma\cup \{0\}))$ if we can show $\Phi_{*,\pm}^{(s)}(x,z)$ is jointly continuous on $X\times(\C\setminus (\Sigma\cup \{0\}))$, which is the consequence of the following lemma.

\begin{lemma}\label{Jointlycontinuous}
The functions $\Phi_{*,\pm}^{(s)}(\cdot,\cdot): X\times (\mathbb C\setminus (\Sigma\cup \{0\}))\rightarrow\C\PP^1$, $*=f,g$, are jointly continuous.
\end{lemma}
\begin{proof}
For the sake of simplicity, we will just consider $\Phi^{(0)}_{g,-}$. The proof works the same way for other $\Phi_{*,\pm}^{(s)}$, $*=f,g$.  We fix $z_0\in \mathbb C\setminus (\Sigma\cup\{ 0\})$ and $x_0\in X$. To simplify notation, let $\sigma_x=\sigma(\E^{(0)}_{x,-2,-})\subset\partial\D$ be the spectrum, and $\mu_x=\mu^{(0)}_{g,x,-}$ be the spectral measure.

Let us endow the space of nonempty compact subsets of $\C$ with the Hausdorff metric and endow the space of normal operators with the metric associated to the operator norm. Then, it is a standard result that the map $N\mapsto \sigma(N)$ (spectrum of $N$) is $1$-Lipschitz continuous with respect to these metrics. Since $T$ is an isometry, $\E^{(0)}_{x,-2,-}$ is a continuous function of $x$, where space of unitary operators is given the operator norm topology as above. Indeed, let us show the continuity for $\E_x$, which is clearly a stronger statement. We have for some constant $C=C(\alpha_f,\alpha_g)>0$,
 \begin{align*}
 \|\E_x-\E_y\| & \le \|\mathcal L_x\mathcal M_x-\mathcal L_y\mathcal M_x\|+\|\mathcal L_y\mathcal M_x-\mathcal L_y\mathcal M_y\|\\
 & \le \|\mathcal L_x-\mathcal L_y\|+\|\mathcal M_x-\mathcal M_y\|\\
 & \le C\cdot\sup_{n\in\Z,*\in\{f,g\}}\{|\alpha_*(T^nx)-\alpha_*(T^ny)|+|\rho_*(T^nx)-\rho_*(T^ny)|\}\\
 & \le C\cdot\sup_{*\in\{f,g\}}\{|\alpha_*(x)-\alpha_*(y)|+|\rho_*(x)-\rho_*(y)|\},
 \end{align*}
  which clearly implies $\|\E_x-\E_y\|\to 0$ as $d(x,y)\to0$. Notice the last inequality above follows from the facts: $\alpha_*$ and $\rho_*$ are uniformly continuous on $X$ by compactness of $X$; $T$ is an isometry which implies $d(T^nx,T^ny)=d(x,y)$ for all $n\in\Z$. Then we have the following two consequences: $x\mapsto \sigma_x$ is continuous with respect to the Hausdorff metric; and the map $x\mapsto \mu_x$ is continuous with respect to the weak-$*$ topology.

It is clear that by \eqref{eq.mtoM}, to prove the present lemma it suffices to show the joint continuity of $m^{(0)}_{g,-}(x,z)$ which is defined by \eqref{M_-}. Recall that $z_0$ is either in the resolvent set or in the discrete spectrum of $\E^{(0)}_{x_0,-2,-}$.

If $z_0$ is in the resolvent set of $\E^{(0)}_{x_0,-2,-}$, we may let $\eta:=\mathrm{dist}(z_0,\sigma_{x_0})>0$. Define $B(z,r)=\{z'\in\C: |z'-z|<r\}$ and $B(x,r)=\{x'\in X: d(x',x)<r\}$. Then by the continuity of the spectrum with respect to the Hausdorff metric, there exists a $\delta>0$ such that
\[
\mathrm{dist}(z,\sigma_x)>\frac\eta4, \mbox{ for all } (x,z)\in B(x_0,\delta)\times B(z_0,\frac\eta4).
\]
Hence we obtain the following estimate for all $(x,z)\in B(x_0,\delta)\times B(z_0,\frac\eta4)$:
\begin{align*}
&|m^{(0)}_{g,-}(x,z)-m^{(0)}_{g,-}(x_0,z_0)|\\
\le & |m^{(0)}_{g,-}(x,z)-m^{(0)}_{g,-}(x,z_0)|+|m^{(0)}_{g,-}(x,z_0)-m^{(0)}_{g,-}(x_0,z_0)|\\
\le &\int_{\sigma_{x}}\frac{1}{|\zeta-z|\cdot|\zeta-z_0|}d\mu_{x}\cdot |z-z_0|+\left|\int_{\sigma_{x}\cup \sigma_{x_0}}\frac{\zeta+z_0}{\zeta-z_0}(d\mu_{x}-d\mu_{x_0})\right|\\
\le & \frac{16}{\eta^2}\cdot |z-z_0|+\left|\int_{\sigma_{x}\cup \sigma_{x_0}}\frac{\zeta+z_0}{\zeta-z_0}(d\mu_{x}-d\mu_{x_0})\right|,
\end{align*}
which goes $0$ as $(x,z)$ goes to $(x_0,z_0)$. Indeed, the first term clearly converges to $0$. For the second term, it follows from these facts: $\frac{\zeta+z_0}{\zeta-z_0}$ is a bounded continuous function on $\sigma_{x}\cup \sigma_{x_0}$ for all $x$ in question, and $\mu_{x}$ converges to $\mu_{x_0}$ in the weak-$*$ topology as $x$ approaching $x_0$.

 Next, we consider the case $z_0$ is in the discrete spectrum of $\E^{(0)}_{x_0,-2,-}$. Clearly, in this case it holds that $m^{(0)}_{g,-}(x_0,z_0)=\infty$ and $\gamma:=\mathrm{dist}(z_0,\sigma_{x_0}\setminus\{z_0\})>0$. Let $p\in\Z_+$ be geometric multiplicity of $z_0$, i.e. the dimension of the eigenspace of $z_0$.

 Then by well-known facts about the continuity of a finite system of isolated eigenvalues of finite multiplicity for a norm-continuous family of bounded operators (see, e.g. Section IV.5 in \cite{K76}, or proof of \cite[Lemma 4.1]{marxpreprint}), it holds that for all $x$ sufficiently close to $x_0$, $\sigma_{x}$ may be decomposed as $\sigma'_x\cup \sigma''_x$ where $\sigma'_x$ is close in Hausdorff metric to $\sigma_{x_0}\setminus\{z_0\}$, and $\sigma''_x$ close in Hausdorff metric to $\{z_0\}$. Moreover, $\sigma''_x$ is a finite set, of which the sum of the geometric multiplicities is less than or equal to $p$. Hence, $\sigma''_x$ is a subset of the discrete spectrum of $\E^{(0)}_{x,-2,-}$ with cardinality less than or equal to $p$.

 By GZ cocycle recursion, it is easy to see that the geometric multiplicity of any eigenvalue is always $1$ in our setting, see e.g. the proof of Corollory~\ref{c:noL1solution} or \cite[Thm 10.16.9]{S2}. Hence, $p=1$ and we may assume $\sigma''_x=\{z_x\}$ for some $z_x\in\partial \D$. Notice that $\lim_{x\rightarrow x_0}z_x=z_0$ since the Hausdorff metric reduces to the usual distance on the complex plane.

Consequently, for $(x,z)$ sufficiently close to $(x_0,z_0)$, we have the following two different cases. The first case is that $z=z_x$, which implies $m^{(0)}_{g,-}(x,z)=\infty$ and we are done. We also have to consider the case where $z$ is in the resolvent set of $\E^{(0)}_{x,-2,-}$. Then as $(x,z)$ approaches $(x_0,z_0)$, it holds for some constant $C,c>0$ depending only on $(x_0,z_0)$ that
\begin{align*}
|m^{(0)}_{g,-}(x,z)|
& = \left|\int_{\sigma_{x}}\frac{\zeta+z}{\zeta-z}d\mu_{x}\right|\\
& \ge \left|\int_{\sigma''_x}\frac{\zeta+z}{\zeta-z}d\mu_x\right|-\left|\int_{\sigma'_{x}}\frac{\zeta+z}{\zeta-z}d\mu_{x}\right|\\
& \ge \left|\int_{\{z_x\}}\frac{\zeta+z}{\zeta-z}d\mu_x\right|-C\cdot\mathrm{dist}(z_0,\sigma_{x_0}\setminus\{z_0\})^{-1}\\
& \ge \frac c{|z_x-z|}-C\cdot\gamma^{-1},
\end{align*}
which goes to $\infty$ since both $z_x$ and $z$ converge to $z_0$, concluding the proof.
\end{proof}

We define the projection
\[
\phi:\mathbb {C}^2\setminus\{0\}\to \mathbb {CP}^1,\ \phi\binom{v_1}{v_2}=\frac{v_1}{v_2}
\]
where $\phi\binom{1}{0}=\infty$. For each $A=\left(\begin{smallmatrix}a&b\\ c&d\end{smallmatrix}\right)\in \mathrm{M}_2(\C)$, let $A\cdot:\C\PP^1\setminus \mathrm{ker} (A)\rightarrow\C\PP^1$ denote the induced map $\phi\circ A\circ \phi^{-1}$. A direct computation shows that
\begin{equation}\label{FLT}
A\cdot z=\frac{az+b}{cz+d}.
\end{equation}

Since subordinate solutions are unique up to a constant multiplier, we have
\begin{equation}\label{transport-}
\left(A^z_f(x)A^z_g(x)\right)\cdot z\Phi_{g,\pm}(x,z)
=
z\Phi_{g,\pm}(Tx,z)
\end{equation}
for all $x\in X_0$ and all $z \in \C \setminus (\Sigma\cup\{0\})$. The following lemma extends this invariance of $z\Phi_{g,-}(x,z)$ to arbitrary $x \in X$.

\begin{lemma}\label{l:m_function_invariance}
For all $z\in \mathbb C\setminus (\Sigma\cup\{0\})$ and for all $x\in X$, it holds that
\begin{equation}\label{eq:notinkernel}
(z\Phi_{g,-}(x,z),1)^\top \notin \mathrm{ker}(A^z_f(x)A^z_g(x)),
\end{equation}
where $(\infty,1)^\top$ is understood as $(1,0)^\top$. Hence, for all $z\in \mathbb C\setminus (\Sigma\cup\{0\})$ and for all $x\in X$, it holds that
\begin{equation}\label{transport}
\left(A^z_f(x)A^z_g(x)\right)\cdot z\Phi_{g,-}(x,z)=z\Phi_{g,-}(Tx,z).
\end{equation}
\end{lemma}

\begin{proof}
We first demonstrate that when $z\in \mathbb C\setminus (\Sigma\cup\{0\})$, $A^z_f(x)A^z_g(x)$ cannot be the zero matrix. We calculate
\[
A^z_f(x)A^z_g(x)
=
\begin{pmatrix}
\alpha_f(x)\overline{\alpha_g(x)}+z^{-1}
& -\alpha_g(x)-z\alpha_f(x)\\
-\overline{\alpha_f(x)}z^{-1}-\overline{\alpha_g(x)} & z+\overline{\alpha_f(x)}\alpha_g(x)
\end{pmatrix}.
\]
Suppose to the contrary that $A^z_f(x)A^z_g(x)$ is the zero matrix; then, it is straightforward to calculate that $|\alpha_f(x)\alpha_g(x)|=1$ and $z=-\overline{\alpha_f(x)}\alpha_g(x)$. But in this case, the matrix $\mathcal E_x-z\mathrm{I}$ has a row and a column that consists entirely of zeroes, and so $z$ must lie in the spectrum of $\mathcal E_x$, contrary to our assumption $z\in\mathbb C\setminus (\Sigma\cup \{0\})$.

It is also easy to check that the kernel of $A^z_f(x)A^z_g(x)$ is trivial if and only if $\alpha_f(x),\alpha_g(x)$ are both in the interior of the unit disk.

Let us first assume $\lvert\alpha_g(x)\rvert=1$. Regardless whether $\lvert \alpha_f(x)\rvert<1 $ or  $\lvert \alpha_f(x)\rvert=1$, we can then calculate that \[\ker(A^z_f(x)A^z_g(x)) =\binom{\alpha_g(z)+z\alpha_f(x)}{\alpha_f(x)\overline{\alpha_g(x)}+z^{-1}}=\binom{z}{\overline{\alpha_g(x)}} =\ker(A^z_g(x)).\]

We first show that $(z\Phi_{-}(x,z),1)^\top \notin \mathrm{ker} (A_g^z(x))$. Since $A_f A_g$ is not the zero matrix, we have
\[
\mathrm{ker}
(A_f^z(x)A_g^z(x))
=
\mathrm{ker}(A_g^z(x))=
\mathrm{span}\begin{pmatrix}
z\alpha_g(x)\\
1
\end{pmatrix}
\]
by a straightforward dimension-counting argument.
Let $\alpha_g=-e^{-is}$. Notice that by \eqref{eq:mfct:solns.alternative} $\Phi_{g,-}(x,z)=-\alpha_g(x)\Phi^{(s)}_{g,-}(x,z)$. Hence
\begin{equation}\label{transversal:to:kernel}
\binom{z\Phi_{g,-}(x,z)}{1}=\begin{pmatrix}
z\alpha_g(x)\\
1
\end{pmatrix}
\iff
\Phi^{(s)}_{g,-}(x,z)=-1.
\end{equation}
By \eqref{eq.mtoM}, $\Phi^{(s)}_{g,-}(x,z)=-1$ implies $m^{(s)}_{g,-}(x,z)=\infty$, which in turn implies that $d\mu^{(s)}_{x,-}$ has a pure point at $z$. Since $\mathcal E_x=\mathcal E^{(s)}_{x,-2,-}\oplus \mathcal E^{(s)}_{x,-1,+}$, this implies that $z$ is an eigenvalue of $\mathcal E_x$, contrary to our assumption that $z \notin \Sigma$.

The calculations proceed in a similar way for the case $\lvert \alpha_f(x)\rvert=1$, $\lvert\alpha_g(x)\rvert<1$. In this case, we have that $A^z_f(x)$ is singular but $A_g^z(x)$ is not. Thus the kernel of $A_f^z(x)A_g^z(x)$ is
\begin{equation}(A_g^z(x))^{-1}\mathrm{ker}(A_f^z(x))=
\mathrm{span}\left\{A_g^z(x)^{-1}\begin{pmatrix}
1\\
\alpha_f(x)
\end{pmatrix}\right\}.\label{fkernel}
\end{equation}

But note that in this case we also have
\[z\Phi_{g,-}(x,z)=(A_g^z(x))^{-1}\cdot \frac{1}{\Phi_{f,-}(x,z)}
\]
by \eqref{GZrecursionOdd}, \eqref{eq:mfct:solns-odd}--\eqref{eq:mfct:solns}.
 Hence $(z\Phi_{g,-}(x),1)\in \ker (A_f^z(x)A_g^z(x))$ is equivalent to
\begin{equation}
\binom{1}{\Phi_{f,-}(x,z)}=\binom{1}{\alpha_f(x)}
\end{equation}

Let $\alpha_f(x)=-e^{-is}$. Notice that by \eqref{eq:mfct:solns.alternative} $\Phi_{f,-}(x,z)=-\alpha_f(x)\Phi^{(s)}_{f,-}(x,z)$. Thus $z\Phi_{g,-}(x)\in \ker A_f^z(x)A_g^z(x)$ is in fact only possible when $\Phi^{(s)}_{f,-}(x,z)=-1$, which again contradicts $z\in \mathbb C\setminus \Sigma$.

Hence we obtain \eqref{eq:notinkernel}, which clearly implies that $(A_f^z(x)A_g^z(x))\cdot z\Phi_{g,-}(x,z)$ defines a continuous map from $X$ to $\C\mathbb P^1$. Notice that $z\Phi_{g,-}(\cdot,z)\in C(X,\C\PP^1)$. Thus \eqref{transport} follows from the fact that $X_0$ is dense in $X$.

\end{proof}

Next we show that $|z\Phi_-(x,z)-z\Phi_+(x,z)|$ is uniformly bounded away from zero.

\begin{lemma}\label{l:unif:trans}
For all $z\in \mathbb C\setminus (\Sigma\cup\{0\})$, there exists a constant $c=c(z)>0$ such that for all $x\in X$,
\begin{equation}\label{eq:unif:trans1}
|z\Phi_-(x,z)-z\Phi_+(x,z)|>c.
\end{equation}
\end{lemma}

\begin{proof}
Let $z\in\mathbb C\setminus (\Sigma\cup\{0\})$ and $x\in X_0$. Let us first note some important differences in convention between \cite{GZ06} and our paper. Firstly, their $\Phi_\pm(z,-2,-2)$ differs with our $\Phi_\pm(x,z)$ by a factor of $z$: refer to their (2.154) and compare our \eqref{eq:mfct:solns-odd}, \eqref{eq:mfct:solns-even}. Furthermore, our $u,v$ are constant multiples of theirs, so the only difference is that factor of $z$. Clearly, by the definition of Schur and anti-Schur functions, it suffices to show \eqref{eq:unif:trans1} for $z \in \partial \D \setminus \Sigma$.

We start with \cite[(3.7)]{GZ06}. We first unpack the definitions from \cite[(2.59), (2.60), (3.5)]{GZ06} to express this formula in terms of (anti-) Carath\'eodory functions in \cite{GZ06}, and then use \cite[(2.151)]{GZ06} to express it in terms of their (anti-) Schur functions. We then replace their (anti-) Schur functions with ours, keeping in mind the differences in convention noted in the previous paragraph. We then obtain
\begin{align*}&\left<
\delta_{-2},(\mathcal E_x-z)^{-1}\delta_{-2}\right>\\
=&\left(1+\frac{1+z\Phi_{g,-}(x,z)}{1-z\Phi_{g,-}(x,z)}\right)\left(1-\frac{1+z\Phi_{g,+}(x,z)}{1-z\Phi_{g,+}(x,z)}\right)\frac{1}{2z\left(\frac{1+z\Phi_{g,+}(x,z)}{1-z\Phi_{g,+}(x,z)}-\frac{1+z\Phi_{g,-}(x,z)}{1-z\Phi_{g,-}(x,z)}\right)}\\
=&\frac{z\Phi_{g,+}(x,z)}{z(z\Phi_{g,-}(x,z)-z\Phi_{g,+}(x,z))}.
\end{align*}
Thus we obtain for all $(x,z)\in X\times \mathbb C\setminus (\Sigma\cup\{0\})$,
\begin{equation}\label{eq:unif:trans2}
\lvert z\Phi_{g,+}(x,z)-z\Phi_{g,-}(x,z)\rvert\geq |\Phi_{g,+}(x,z)|\cdot\mathrm{dist}(z,\Sigma).
\end{equation}
In particular, for $z\in\partial \D\setminus\Sigma$, $|\Phi_{g,+}(x,z)|=1$ which together with \eqref{eq:unif:trans2} implies
$$
\lvert z\Phi_{g,+}(x,z)-z\Phi_{g,-}(x,z)\rvert\geq \mathrm{dist}(z,\Sigma)>0,
$$
concluding the proof.
\end{proof}
Now we have all the information that we need to prove \eqref{eq:DSawayspectrum}.
\begin{proof}[Proof of \eqref{eq:DSawayspectrum}]
Let $A^z(x) = A^z_f(x)A^z_g(x)$ and $M^z(x)=\frac{1}{\rho_f(x)\rho_g(x)}A^z(x)$, the original $2$-step GZ cocycle. Notice that this $M^z$ coincides with $M(;z)$ in previous sections. Let $L(T,z)=L(T,M^z)$. Since $M^z(x)\in \mathrm{SL}(2,\C)$ for all $x\in X_0$, $(T,M^z)$ has non-degenerate Lyapunov spectrum if and only if $L(T,z)>0$. By the relation between $A^z$ and $M^z$, $(T,A^z)$ has nondegenerate Lyapunov spectrum if and only if $(T,M^z)$ has nondegenerate Lyapunov spectrum, hence, if and only if $L(T,z)>0$. Thus, using a CMV version of the Combes--Thomas estimate (e.g., \cite[Theorem~10.14.2]{S2}) %(\ref{c:nondeg:LE}) in the Appendix,
$(T,A^z)$ has non-degenerate Lyapunov spectrum for all $z\in\C\setminus(\Sigma\cup\{0\})$. Throughout this proof, we will fix such a $z$.

We let $E^+(.,z)$ be the continuous lift of $z\Phi_{g,-}(.,z)$ and let $E^-(.,z)$ be a continuous lift of $z \Phi_{g,+}(.,z)$ to subspaces of $\mathbb C^2$. By Lemma~\ref{Jointlycontinuous}, formula \eqref{transport-}, Lemma~\ref{l:m_function_invariance} and \ref{l:unif:trans}, we obtain all $x\in X$, it holds that
\begin{align}
\label{eq:inv:E-}&A^z(x)E^-(x,z)\subseteq E^-(Tx,z),\\
\label{eq:inv:E+}&A^z(x)E^+(x,z)= E^+(Tx,z).
\end{align}

Then we define $Q^z(x)$ to be:
$$
\begin{pmatrix}z & z\Phi_{g,+}(x,z)\\ \frac{1}{\Phi_{g,-}(x,z)} &1\end{pmatrix},\ z\in\D;\mbox{ or }
\begin{pmatrix}z\Phi_{g,-}(x,z) & z\\ 1& \frac{1}{\Phi_{g,+}(x,z)}\end{pmatrix},\ z\notin\D.
$$

Then since $\Phi_{g,+}$ is a Schur function and $\Phi_{g,-}$ is an anti-Schur function, we have
$$
Q^z\in C(X,\mathrm{M}_2(\C)).
$$
Next we show $Q^z(x)$ is nonsingular for all $z\in \C\setminus(\Sigma\cup\{0\})$ and all $x\in X$. Indeed, if $z\in\D$, then Lemma~\ref{l:unif:trans} implies
$$
|\det(Q^z(x))|=\frac1{|\Phi_{g,-}(x,z)|}|z\Phi_{g,-}(x,z)-z\Phi_{g,+}(x,z)|\neq 0
$$
for all $x\in X$ whenever $\Phi_{g,-}(x,z)\neq\infty$. But $\Phi_{g,-}(x,z)=\infty$ implies $\det(Q^z(x))=|z|\neq 0$ since we assumed $z\neq 0$. If $z\notin\D$, then by \eqref{eq:unif:trans2},
$$
|\det(Q^z(x))|=\frac{1}{|\Phi_{g,+}(x,z)|}|z\Phi_{g,-}(x,z)-z\Phi_{g,+}(x,z)|>\mathrm{dist}(z,\Sigma)>0.
$$
In any case, we obtain
\begin{equation}\label{eq:GL:conjugacy}
Q^z\in C(X,\mathrm{GL}(2,\C)).
\end{equation}

Now by \eqref{eq:inv:E-} and \eqref{eq:inv:E+}, we must have for some $\lambda_i(x,z)$, $i=1,2$,
\begin{equation}\label{eq:conjugacy}
Q^z(Tx)^{-1}A^z(x)Q^z(x)=\begin{pmatrix}\lambda_1(x,z)& 0\\ 0 & \lambda_2(x,z)\end{pmatrix}.
\end{equation}
Since $Q^z\in C(X,\mathrm{GL}(2,\C))$ and $A^z\in C(X,\mathrm M_2(\C))$, it holds that
$$
\lambda_k(\cdot,z) \in C(X,\C),\ k=1,2.
$$
Furthermore, for each $x\in X_0$, we have
$$
\binom{z\Phi_{g,-}(x,z)}{1}\in \mathrm{span}\binom{u_{x,-}(-2)}{v_{x,-}(-2)},\ \binom{z\Phi_{g,+}(x,z)}{1}\in \mathrm{span}\binom{u_{x,+}(-2)}{v_{x,+}(-2)},
$$
where $\binom{u_{x,\pm}}{v_{x,\pm}}\in\left(\ell^2(\Z_\pm)\right)^2$. Thus, by the discussion of last part of Section~\ref{sec:dominated splitting}, it must hold that for $k=1,2$,
$$
L_k(T,A^z)
=
\int_X\log|\lambda_k(x,z)| \, d\mu(x).
$$
In particular, $\int_X\log|\lambda_1(x,z)| \, d\mu(x) > \int_X\log|\lambda_2(x,z)| \, d\mu(x)$ since $(T,A^z)$ admits nondegenerate Lyapunov spectrum for $z\in\C\setminus(\Sigma\cup\{0\})$. Notice that by definition of $\lambda_1(x)$, it holds that
$$
A^z(x)\binom{z}{\frac{1}{\Phi_{g,-}(x,z)}}=\lambda_1(x,z)\binom{z}{\frac{1}{\Phi_{g,-}(Tx,z)}},\ z\in \D,
$$
and
$$
A^z(x)\binom{z\Phi_{g,-}(x,z)}{1}=\lambda_1(x,z)\binom{z\Phi_{g,-}(Tx,z)}{1}, z\notin\D.
$$
By \eqref{eq:notinkernel} and compactness of $X$, we must have
$$
\inf_{x\in X}|\lambda_1(x,z)|>c>0,
$$
which implies $\log|\lambda_1(\cdot,z)|\in C(X,\R)$. Now unique ergodicity of $T$ implies that
$$
\frac1n\sum^{n-1}_{i=0}\log|\lambda_1(T^ix)| \mbox{ converges to } L_1(T,A^z) \mbox{ uniformly in } x\in X.
$$
Hence for each $\gamma>0$, there exists a $N_1$ such that for all $n\ge N_1$ and for all $x\in X$
\begin{equation}\label{eq:up}
\frac1n\sum^{n-1}_{i=0}\log|\lambda_1(T^ix)|>L_1(T,A^z)-\gamma.
\end{equation}
On the other hand since $\lambda_1\in C(X,\C)$, unique ergodicity of $T$ implies that for each $\gamma>0$, there exists a $N_2>0$ such that for all $n\ge N_2$, it holds uniformly for all $x\in X$:
\begin{equation}\label{eq:down}
\frac1n\sum^{n-1}_{i=0}\log|\lambda_2(T^ix)|<L_2(T,A^z)+\gamma,
\end{equation}
see e.g. \cite{Furman}\footnote{In \cite{Furman}, the author assumed continuity of the subadditive sequence of functions. However, the proof works without any change for upper semi-continuous functions. In particular, it works for our $\log|\lambda_2(x)|$.}. By choosing $\gamma<\frac12(L_1(T,A^z)-L_2(T,A^z))$ and set $N=\max\{N_1,N_2\}$, combining \eqref{eq:up} and \eqref{eq:down}, we then obtain uniformly for all $x\in X$,
$$
\prod^{N-1}_{i=0}|\lambda_1(T^ix)|>\prod^{N-1}_{i=0}|\lambda_2(T^ix)|,
$$
concluding the proof that $(T,A^z)\in\mathcal D\mathcal S$.
\end{proof}

\subsection{Absence of dominating splitting on the Spectrum} \label{ssec:absenceofDSonspec}

In this section, we prove
\[
\Sigma
\subseteq
\{ z \in \partial \D : (T,A^z) \notin \mathcal{DS}\}.
\]
Since $\E_x$ is a unitary operator, we have $\Sigma\subset\partial\D$. We then instead show a equivalent result:
\[
\C\setminus\Sigma
\supset
\{ z \in \C\setminus\{0\} : (T,A^z) \in \mathcal{DS}\}.
\]

Suppose that $(T,A^z) \in \mathcal{DS}$. Then there exists a $Q^z\in C(X,\mathrm{GL}(2,\C))$ such that
$$
Q^z(Tx)^{-1}A^z(x)Q^z(x)
=
\begin{pmatrix}
\lambda_1(x,z)& 0\\
0 & \lambda_2(x,z)
\end{pmatrix}.
$$
Hence, we obtain
$$
Q^z(Tx)^{-1}M^z(x)Q^z(x)
=
\begin{pmatrix}
\Lambda_1(x,z)& 0\\
0 & \Lambda_2(x,z)
\end{pmatrix}.
$$
where $\Lambda_k(x)=\frac{\lambda_k(x)}{\rho_f(x)\rho_g(x)}$, $k=1,2$. Let $L_k(T,z)$, $k=1,2$, be the two Lyapunov exponents of $(T,M^z)$. In particular, $L_1(T,z)=L(T,z)$. Thus, by the discussion of the last part of Section~\ref{sec:dominated splitting}, it holds that $L_k(T,z)=\int_X\log|\Lambda_k(x,z)| \, d\mu(x)$ and
$$
L(T,z)
=
\int_X\log|\Lambda_1(x,z)| \, d\mu(x)
=
-\int_X\log|\Lambda_2(x,z)| \, d\mu(x)
>
0.
$$
Let $\vec m^z_k(x)$ be the column vectors of $Q^z(x)$, $k=1,2$. Since $Q^z(x)\in\mathrm{GL}(2,\C)$, $\vec m^z_1(x)$ and $\vec m^z_2(x)$ are linearly independent. Moreover, for each $x\in X_0$
$$
\lim_{n\rightarrow\infty} \frac1n \log\|M^z_n(x)\vec m^z_1(x)\|
=
L_1(T,z)
=
\lim_{n\rightarrow\infty} \frac1n \log\|M^z_{-n}(x)\vec m^z_2(x)\|
$$
Clearly for each $x \in X$ and for each $\vec w\in\C^2$, we may decompose $\vec w$ as
$$
\vec w=c_1\vec m^z_1(x)+c_2\vec m^z_2(x).
$$
Hence for each $x\in X_0$ and all $\vec w \in\C^2$, $\|M^z_n(x)\vec w\|$ grows exponentially either as $n\rightarrow\infty$ or as $n\rightarrow -\infty$. Thus, by the GZ recursions in \eqref{GZrecursionOdd}, \eqref{GZrecursionEven}, and the fact that $\mathcal M_xu=zv$, we see that $z$ cannot be a generalized eigenvalue of $\E_{x}$ for any $x\in X_0$; recall that a generalized eigenvalue of $\E_x$ is a spectral parameter $z\in \C$ at which $\E_x u=z u$ enjoys a nontrivial polynomially bounded solution in $\C^{\Z}$. Let $\Sigma_{\mathrm g}(\E_x)$ denote the set of generalized eigenvalues of $\E_x$. From the CMV version \cite[Theorem~6]{DFLY2} of a theorem of Sch'nol-Berezanskii \cite{Bere, Sch}, it holds that
$$
\overline{\Sigma_{\mathrm g}(\E_x)}=\Sigma.
$$
Since $\mathcal{DS}$ defines an open subset of $C(X,\mathrm{M}_2(\C))$, it follows that $z\notin\Sigma$. This concludes the proof of Theorem~\ref{t:johnson}.

\end{document}